\documentclass[reqno,11pt]{article}

\usepackage{geometry}
\usepackage{graphicx}
\usepackage{fancyhdr,mathtools,footnote}
\usepackage{amsmath,color}
\usepackage{amsthm}
\usepackage{amsfonts}
\newtheorem{theorem}{Theorem}[section]
\newtheorem{lemma}[theorem]{Lemma}
\theoremstyle{definition}
\newtheorem{definition}[theorem]{Definition}

\newtheorem{remark}[theorem]{Remark}

\date{}
\def\m@th{\mathsurround=0pt}
\def\eqal#1{\null\,\vcenter{\openup\jot\m@th
 \ialign{\strut\hfil$\displaystyle{##}$&&$\displaystyle{{}##}$\hfil
 \crcr#1\crcr}}\,}
\def\matrix#1{\null\,\vcenter{\normalbaselines\m@th
 \ialign{\hfil$##$\hfil&&\quad\hfil$##$\hfil\crcr
 \mathstrut\crcr\noalign{\kern-\baselineskip}
 #1\crcr\mathstrut\crcr\noalign{\kern-\baselineskip}}}\,}
\def\N{{\mathbb N}}
\def\R{{\mathbb R}}
\def\divv{{\rm div}\,}

\def\bye{\end{document}}

\newcommand\blfootnote[1]{%
  \begingroup
  \renewcommand\thefootnote{}\footnote{#1}%
  \addtocounter{footnote}{-1}%
  \endgroup
}
\newcommand{\lec}{\lesssim}

\newcommand{\eqnb}{\begin{equation}}
\newcommand{\eqne}{\end{equation}}

\newcommand{\p}{\partial}
\newcommand{\ta}{\tau}
\renewcommand{\div}{\mathrm{div}}
\renewcommand{\d}{\mathrm{d}}

\numberwithin{equation}{section}

\title{On the regularity of axially-symmetric solutions to the incompressible Navier-Stokes equations in a cylinder}
\author{Wojciech S. O\.za\'nski, Wojciech M. Zaj\c{a}czkowski}

\begin{document}

\input amssym.def
\input amssym.tex
\maketitle
\thispagestyle{fancy}

\blfootnote{
\noindent WO: Department of Mathematics, Florida State University, Tallahassee, FL 32301, email: wozanski@fsu.edu}
\blfootnote{\noindent WZ: Institute of Mathematics (emeritus professor), Polish Academy of Sciences, \'Sniadeckich 8, 00-656 Warsaw, Poland, and Institute of Mathematics and Cryptology, Cybernetics Faculty, Military University of Technology, S. Kaliskiego 2, 00-908 Warsaw, Poland, e-mail: wz@impan.pl}

\begin{abstract}
We consider the axisymmetric Navier-Stokes equations in a finite cylinder $\Omega\subset\mathbb{R}^3$. We assume that $v_r$, $v_\varphi$, $\omega_\varphi$ vanish on the lateral boundary $\partial \Omega$ of the cylinder, and that $v_z$, $\omega_\varphi$, $\p_z v_\varphi$ vanish on the top and bottom parts of the boundary $\partial \Omega$, where we used standard cylindrical coordinates, and we denoted by $\omega =\mathrm{curl}\, v$ the vorticity field. We use weighted estimates and $H^3$ Sobolev estimate on the modified stream function to derive three order-reduction estimates. These enable one to reduce the order of the nonlinear estimates of the equations, and help observe that the solutions to the equations are ``almost regular''. We use the order-reduction estimates to show that the solution to the equations remains regular as long as, for any $p\in (6,\infty)$, $\| v_\varphi \|_{L^\infty_t L^p_x}/\| v_\varphi \|_{L^\infty_t L^\infty_x}$  remains  bounded below by a positive number.
\end{abstract}


\noindent
{\bf keywords:} Navier-Stokes equations, axially-symmetric solutions, cylindrical domain

\section{Introduction}\label{s1}

We are concerned with the $3$D incompressible Navier-Stokes equations,
\eqnb\label{nse}
\begin{split}
\p_t v - \nu \Delta v + (v\cdot \nabla )v + \nabla p &=f,\\
\div\, v&=0 \quad \text{ in }\Omega^T,
\end{split}
\eqne
under the axisymmetry constraint, where $\Omega^T \coloneqq \Omega \times (0,T)$, $T>0$, $v=v(x,t)\in \R^3$ denotes the velocity field, $p=p(x,t)\in \R$ denotes the pressure function, $f=f(x,t)\in \R^3$ denotes the external force field, and $\nu >0$ denotes the viscosity. As for $\Omega$ we focus on the case of a finite cylinder,
\[
\Omega=\left\lbrace x\in\R^3\colon x_1^2+x_2^2<R^2,|x_3|<a\right\rbrace,
\]
where $a,R>0$ are constants. We note that 
\[
S\coloneqq \p \Omega = S_1 \cup S_2,
\]
where
\[
\begin{split}
S_1&=\left\lbrace x\in\R^3\colon\sqrt{x_1^2+x_2^2}=R,\ x_3\in [-a,a]\right\rbrace ,\\
S_2&=\left\lbrace x\in\R^3\colon\sqrt{x_1^2+x_2^2}<R,\ x_3\in\{-a,a\}\right\rbrace 
\end{split}
\]
denote the lateral boundary and the top and bottom parts of the boundary, respectively.

In order to state the boundary conditions  stating our main result we describe the problem in cylindrical coordinates, that is we will write 
\[
x_1=r\cos\varphi,\ \ x_2=r\sin\varphi,\ \ x_3=z,
\]
and we will use standard cylindrical unit vectors, so that, for example, 
\[
v=v_r \bar e_r + v_\varphi\bar e_\varphi + v_z\bar e_z.
\]
We will denote partial derivatives by using the subscript comma notation, e.g. 
\[
v_{r,z} \coloneqq \p_z v_r.
\]
We assume the boundary conditions 
\eqnb\label{bcs}
\begin{split}
&v_r=v_\varphi = \omega_\varphi =0 \qquad \text{ on }  S_1^T=S_1\times(0,T),\\
&v_z = \omega_\varphi =  v_{\varphi,z} =0 \qquad \text{ on }   S_2^T=S_2\times(0,T),
\end{split}
\eqne
where $\omega \coloneqq \mathrm{curl} \, v$ denotes the vorticity vector and we assume initial condition $v(0)=v_0$, where $v_0$ is a given divergence-free vector field satisfying the same boundary conditions. We note that such boundary conditions have first appeared in the  work of Ladyzhenskaya \cite{L1}. In a sense, the boundary conditions \eqref{bcs} are natural, since, when considering the vorticity-stream function formulation (see \eqref{1.14} below), boundary terms resulting from the energy estimates will naturally involve $\omega_\varphi|_{S}$. This, together with the no-penetration condition naturally lead to \eqref{bcs}.

We will also denote the \emph{swirl} by
\[
u\coloneqq r v_\varphi .
\label{swirl}
\]
 Note that 
\begin{equation}\eqal{
&\omega_r=-v_{\varphi,z}=-{1\over r}u_{,z},\cr
&\omega_\varphi=v_{r,z}-v_{z,r},\cr
&\omega_z={1\over r}(rv_\varphi)_{,r}=v_{\varphi,r}+{v_\varphi\over r}={1\over r}u_{,r},\cr}
\label{1.13}
\end{equation}
so that the boundary conditions \eqref{bcs} imply in particular that
\eqnb\label{bcs1}
\begin{split}
&\omega_r = v_{z,r} =u=0, \,\, \omega_z=v_{\varphi,r} \qquad \text{ on }  S_1^T,\\
&\omega_r = v_{r,z} =\omega_{z,z} =u_{,z}=0 \qquad \text{ on }   S_2^T,
\end{split}
\eqne
The Navier-Stokes equations \eqref{nse} become, in cylindrical coordinates, 
\eqnb\label{nse_cylin}
\begin{split}
v_{r,t}+v\cdot\nabla v_r-{v_\varphi^2\over r}-\nu\Delta v_r+\nu{v_r\over r^2}&=-p_{,r}+f_r,\\
v_{\varphi,t}+v\cdot\nabla v_\varphi+{v_r\over r}v_\varphi-\nu\Delta v_\varphi+\nu{v_\varphi\over r^2}&=f_\varphi,\\
v_{z,t}+v\cdot\nabla v_z-\nu\Delta v_z&=-p_{,z}+f_z,\\
(rv_r)_{,r}+(rv_z)_{,z}&=0,
\end{split}
\eqne
where
\[v\cdot\nabla=(v_r\bar e_r+v_z\bar e_z)\cdot\nabla=v_r\partial_r+v_z\partial_z,\qquad \Delta u={1\over r}(ru_{,r})_{,r}+u_{,zz}.
\]

On the other hand, the vorticity formulation becomes 
\eqnb\label{1.9}
\begin{split}
\omega_{r,t}+v\cdot\nabla\omega_r-\nu\Delta\omega_r+\nu{\omega_r\over r^2}&=\omega_rv_{r,r}+\omega_zv_{r,z}+F_r,\\
\omega_{\varphi,t}+v\cdot\nabla\omega_\varphi-{v_r\over r}\omega_\varphi-\nu\Delta\omega_\varphi+\nu{\omega_\varphi\over r^2}&={2\over r}v_\varphi v_{\varphi,z}+F_\varphi,\\
\omega_{z,t}+v\cdot\nabla\omega_z-\nu\Delta\omega_z &=\omega_rv_{z,r}+\omega_zv_{z,z}+F_z,
\end{split}
\eqne
where $F\coloneqq \mathrm{curl}\, f$ and the swirl equation is
\eqnb\label{u_eq}
u_{,t}+v\cdot\nabla u-\nu\Delta u+{2\nu\over r}u_{,r} =rf_\varphi=:  f_0.
\eqne
We will use the notation
\begin{equation}
(\Phi,\Gamma)=\left( \frac{\omega_r}r,\frac{\omega_\varphi}r\right), 
\label{1.16}
\end{equation}
and we note that $\Phi$, $\Gamma$ satisfy 
\begin{align}
\Phi_{,t}+v\cdot\nabla\Phi-\nu\bigg(\Delta+{2\over r}\partial_r\bigg)\Phi-(\omega_r\partial_r+\omega_z\partial_z){v_r\over r}=F_r/r\equiv\bar F_r, \label{1.17}\\ 
\Gamma_{,t}+v\cdot\nabla\Gamma-\nu\bigg(\Delta+{2\over r}\partial_r\bigg)\Gamma+2{v_\varphi\over r}\Phi=F_\varphi/r\equiv\bar F_\varphi, \label{1.18}
\end{align}
recall \cite[(1.6)]{CFZ}. Moreover, by  \eqref{bcs}, \eqref{bcs1}, $\Gamma$ and $\Phi$ satisfy the boundary conditions 
\eqnb\label{bcs_PhiGamma}
\Phi=\Gamma=0 \qquad \text{ on }S^T.
\eqne
We note that $\eqref{nse_cylin}_4$ implies existence of the \emph{stream function} $\psi$ which solves the problem
\eqnb
\label{1.14}
\begin{split}
-\Delta\psi+{\psi\over r^2}&=\omega_\varphi,\\
\psi|_S&=0.
\end{split}
\eqne
Then $v$ can be expressed in terms of the stream function,
\begin{equation}\eqal{
&v_r=-\psi_{,z},\hspace{2cm} v_z={1\over r}(r\psi)_{,r}=\psi_{,r}+{\psi\over r},\cr
&v_{r,r}=-\psi_{,zr},\hspace{1.65cm} v_{r,z}=-\psi_{,zz},\cr
&v_{z,z}=\psi_{,rz}+{\psi_{,z}\over r},\hspace{0.85cm} v_{z,r}=\psi_{,rr}+{1\over r}\psi_{,r}-{\psi\over r^2}.\cr}
\label{1.15}
\end{equation}
We will also use the \emph{modified stream function},
\begin{equation}
\psi_1\coloneqq \frac{\psi}{r},
\label{1.21}
\end{equation}
which satisfies
\eqnb\label{psi1_eq}
\begin{split}-\Delta\psi_1-{2\over r}\psi_{1,r}&=\Gamma,\\
\left. \psi_1 \right|_S&=0,
\end{split}
\eqne
and lets us express $v$ as 
\begin{equation}\eqal{
&v_r=-r\psi_{1,z},\hspace{2cm} v_z=(r\psi_1)_{,r}+\psi_1=r\psi_{1,r}+2\psi_1,\cr
&v_{r,r}=-\psi_{1,z}-r\psi_{1,rz},\hspace{0.5cm} v_{r,z}=-r\psi_{1,zz},\cr
&v_{z,z}=r\psi_{1,rz}+2\psi_{1,z},\hspace{0.6cm} v_{z,r}=3\psi_{1,r}+r\psi_{1,rr}.\cr}
\label{1.22}
\end{equation}
We note that the boundary conditions \eqref{bcs} imply that
\eqnb\label{bc_psi1,zz}
\psi_{1,zz} =0 \qquad \text{ on } S_2.
\eqne
Indeed, restricting the PDE in \eqref{psi1_eq} onto $S_2$ gives $\psi_{1,zz}=-\Gamma$ on $S_2$, and recalling \eqref{bcs} that $\Gamma|_{S_2}=0$, by \eqref{bcs}, gives \eqref{bc_psi1,zz}.\\

We emphasize that the global well-posedness of the equations (either in the above setting or on $\R^3$) remains an important open problem. We only note a few references on the topic \cite{CFZ,KP,NZ,NZ1,NP1,NP2,OP} and we refer the reader to \cite{Z1,Z2} for further discussion. One of the main properties of the axisymmetric setting of the equations is the maximum principle \eqref{2.9} for the swirl. We note that this is critical in the sense that if $v_\lambda (x,t) \coloneqq \lambda v (\lambda x , \lambda^2 t)$ then $\sup |r v_\lambda |=\sup |r v |$. In a sense, one of the main focuses of the analysis of the axisymmetric Navier-Stokes equations is to make the best possible use of the maximum principle and other properties (such as the energy estimates, Lemma~\ref{l5.1}) of the swirl. 

On the other hand, it was demonstrated by \cite{CFZ} that the solution $v$ is controlled by the energy norm of $\Phi$, $\Gamma$,
\eqnb\label{def_V}
X(t) \coloneqq \| \Phi \|_{V(\Omega^t)} + \|\Gamma \|_{V(\Omega^t)} \qquad \text{ where }\|w\|_{V(\Omega^t)}\coloneqq |w|_{2,\infty,\Omega^t}+|\nabla w|_{2,\Omega^t},
\eqne
who also showed that $X$ is under control provided $v_\varphi$ is sufficiently small near the axis. This has recently been developed in \cite{OP} to a setting that allows quantitative control of the solution $v$ and all its derivatives in terms of the weak $L^3$ norm of $v$ only. Here, and below, $|\cdot |_{p,q ,\Omega^t } = \| \cdot \|_{L_q (0,t;L_p(\Omega ))}$ and $|\cdot |_{p,\Omega^t} = |\cdot |_{p,p,\Omega^t}$ denote space-time Lebesgue spaces, see \eqref{def_Lp}. \\

\noindent Our first result is concerned with another quantitative control using $X$.
\begin{theorem}[Order-reduction estimates]\label{T1}
Let $v$ be a regular solution of \eqref{nse}--\eqref{bcs} on $(0,T)$. Then  
\eqnb\label{est1}
| \Phi |_{2,\Omega^t } \lec_{\delta } (1+|v_\varphi |_{\infty, \Omega^t }^{\delta} )  X^{\frac12} +1
\eqne
for every $\delta\in (0,1)$, $t\in (0,T)$. Moreover,
\eqnb\label{est2}
|v_\varphi |_{\infty , \Omega^t} \lec  X^{3/4} + 1.
\eqne
and, if $d\in (3,\infty )$ is such that
\begin{equation}
\frac{|v_\varphi|_{d,\infty,\Omega^t} }{|v_\varphi|_{\infty,\Omega^t}} \ge c_0
\label{4.29}
\end{equation}
for some $c_0>0$, then
\eqnb\label{est3}
| v_\varphi |_{d,\infty,\Omega^t} \lec_{d,\delta  } {c_0^{1-\frac{d}{2-\delta  }}}  \left( X(t)^{\frac{1}{4-2\delta  }} +1\right).
\eqne 
\end{theorem}
We note that the proof of  Theorem~\ref{T1}, relies on the weighted estimates of the modified stream function $\psi_1$ introduced by \cite{NZ} (see also Lemma~\ref{l2.18}), new $H^2$ and $H^3$ elliptic estimates on $\psi_1$ (see Lemmas~\ref{l3.1} and \ref{l3.2}), and swirl estimates (see Lemma~\ref{l5.1}).

We emphasize that the estimates \eqref{est1}--\eqref{est3} are valid for regular solutions. In particular, the estimates use the fact that smooth $v$ admits the following expansions near the axis
\begin{equation}
v_r(r,z,t)=a_1(z,t)r+a_2(z,t)r^3+\cdots,
\label{2.19}
\end{equation}
\begin{equation}
v_\varphi(r,z,t)=b_1(z,t)r+b_2(z,t)r^3+\cdots,
\label{2.20}
\end{equation}
\begin{equation}
\psi(r,z,t)=d_1(z,t)r+d_2(z,t)r^3+\cdots\ 
\label{2.21}
\end{equation}
In particular
\begin{equation}
\psi_1(r,z,t)=d_1(z,t)+d_2(z,t)r^2+\cdots,
\label{2.22}
\end{equation}
\begin{equation}
\psi_{1,r}(r,z,t)=2d_2(z,t)r+\cdots\qquad\qquad
\label{2.23}
\end{equation}
which was shown by Liu \& Wang \cite[Corollary~1 and Lemma~2]{LW}. For example, for \eqref{2.19}--\eqref{2.23} it suffices that $v\in W_2^{4,2}(\Omega^t)$, $\nabla p\in W_2^{2,1}(\Omega^t)$. We show in Section~\ref{sec_reg} that this is true as long as $X(t)$ remains bounded. \\

We emphasize that, in \eqref{est1}, the $L^2$ norm of $\Phi$ is controlled by (almost) square root of $X$, while \eqref{est3} lets us control $|v_\varphi |_{d,\infty}$ in terms of (almost) the fourth root of $X$.
In particular, estimates \eqref{est1}--\eqref{est3} allow us to analyze certain energy norms of the solution $v$ of the axisymmetric Navier-Stokes equations \eqref{nse} in the way where the order of the nonlinear estimates can be reduced. In particular, in our second result, we show that condition \eqref{4.29} implies regularity of solution.

\begin{theorem}[Conditional regularity]\label{t1.3}
Let $v$ be a smooth solution on time interval $(0,T)$. Suppose that there exist $d\in (6,\infty)$ and a constant $c_0>0$ such that \eqref{4.29} holds on $(0,T)$.  Then 
\eqnb\label{X_bd}
X(t) \lec_{\delta , d,c_0} 1
\eqne
for $t\in (0,T)$.
\end{theorem}
We emphasize that the implicit constant in \eqref{X_bd} does not depend on time. Here, and throughout the paper, we use the notation ``$\lec_{a,f}$'' to denote ``$\leq C$'', where $C>0$ depends on $a$ and $f$''.

We note that this work is a continuation of \cite{Z1, Z2}, where a Serrin type result was proved, which is a consequence of the fact that $\psi_1=\psi/r$, where $\psi$ is the stream function and $r$ is the radius, must vanish on the axis of symmetry. Vanishing of $\psi_1|_{r=0}$ in \cite{Z1, Z2} follows from exploiting weighted Sobolev spaces with elements vanishing on the axis of symmetry. In this paper we were able to drop the restriction $\psi_1|_{r=0}=0$ because such weighted spaces were not used.

However, the properties are not sufficient to derive estimates \eqref{psi1_h2} and \eqref{3.13}--\eqref{3.15}, which are crucial in controlling $X(t)$. To prove these estimates we make use of the expansions \eqref{2.19}--\eqref{2.23}, due to Liu and Wang \cite{LW}.

We recall (from the classical work of Ladyzhenskaya \cite[Chapter~6, Section~4, Theorem~4]{L2}, see also \cite{L1} and the partial regularity theory of \cite{CKN}) that any singularity of axisymmetric solutions to \eqref{nse} must occur on the axis of symmetry.  The methods presented in this paper make use of regular solutions, for which there are no singularities at the axis, and so the expansions \eqref{2.19}--\eqref{2.23} are valid. Whether it is possible to control $X(t)$ without exploiting these expansions, remains an interesting open problem.

\subsection{Proof of Theorem~\ref{t1.3}}
We use an energy estimate on $\Phi$, $\Gamma$ to obtain
\eqnb\label{005}
X(t)^2 \lec_{\delta } 1+ I (1+ |v_\varphi |^{\delta }_{\infty, \Omega^t } )
\eqne
where
\eqnb\label{def_of_I}
I\coloneqq  \left| \int_{\Omega^t } \frac{v_\varphi }r \Phi \Gamma \right| .
\eqne
and the implicit constant (here and below) depends on the initial data $v_0$ and constants $D_1,\ldots , D_{12}$, see \eqref{4.1} for details (where the constant is quantified).

 We show in \eqref{4.12} that
\eqnb\label{I_control}
I \lec_{\varepsilon_1, \varepsilon_2} |v_{\varphi}|_{d,\infty,\Omega^t}^\varepsilon |\Phi |^\theta_{2,\Omega^t} X(t)^{2-\theta}
\eqne
for every $d>3$, given $\varepsilon_1, \varepsilon_2>0$ are sufficiently small such that
\eqnb\label{temp1}
\theta\coloneqq \big(1-{3\over d}\big)\varepsilon_1-{3\over d}\varepsilon_2\in (0,1), \quad 1+\frac{\varepsilon_2}{\varepsilon_1} \leq \frac{d}3,
\eqne
and $\varepsilon \coloneqq \varepsilon_1+\varepsilon_2$. We will fix $\varepsilon_1, \varepsilon_2$ below, but we see at this point that achieving \eqref{temp1} is possible, by first picking $\varepsilon_1$ small enough so that $(1-3/d)\varepsilon_1 <1$ (which is possible, as $d>3$), and then taking  $\varepsilon_2 $ sufficiently small so that $\theta >0$ and $\varepsilon_2 < \varepsilon_1 (d/3-1)$, which gives \eqref{temp1}.

Applying \eqref{est1} in \eqref{I_control}, and plugging this in \eqref{005} we thus obtain 
\eqnb\label{01}
\begin{split}
X(t)^2 &\lec_{\varepsilon_1, \varepsilon_2, \delta } 1+  (1+|v_\varphi |_{\infty, \Omega^t }^{\delta } )^\theta  (1+ |v_\varphi |^{\delta }_{\infty, \Omega^t } )|v_{\varphi}|_{d,\infty,\Omega^t}^\varepsilon X(t)^{2-\theta/2}\\
&\lec_{\varepsilon_1, \varepsilon_2, \delta } 1+  \left( 1+|v_\varphi |_{\infty, \Omega^t }^{2\delta  \theta } \right)|v_{\varphi}|_{d,\infty,\Omega^t}^\varepsilon X(t)^{2-\theta/2}.
\end{split}
\eqne
We now let $d\in (6,\infty)$ be such that the assumption \eqref{4.29} holds, and we apply the order-reduction estimate \eqref{est3} for $|v_\varphi |_{d,\infty,\Omega^t }$ in \eqref{01} to obtain
\eqnb\label{003}
X(t)^2 \lec_{\varepsilon_1, \varepsilon_2 , \delta , c_0, d} 1+  X(t)^{2-\frac{\theta}2 + \frac{\varepsilon}{4-2\delta } + \frac{3\delta  \theta }2},
\eqne
where we also used the order-reduction estimate \eqref{est2} for $|v_\varphi |_{\infty,\Omega^t}$.
We now want to determine $\varepsilon_1,\varepsilon_2,\delta $ such that the last power is strictly less than $2$, so that the last term can be absorbed by the left-hand side. To this end, we note that $\theta$ (defined in \eqref{temp1}) satisfies  
\eqnb\label{002}
\theta > \frac{\varepsilon}2\quad \text{ if and only if } \quad \varepsilon_2  < \frac{d-6}{d+6} \varepsilon_1.
\eqne
Thus, we first fix $\varepsilon_1>0$  small enough such that $(1-3/d)\varepsilon_1 < 1$ (so that $\theta <1$). We then take $\varepsilon_2 >0$ sufficiently small so that both \eqref{temp1} and \eqref{002} hold. We note that this fixes $\theta\in (0,1)$ and also guarantees that
\[
2-\frac{\theta}2 + \frac{\varepsilon}4 <2.
\]
Therefore, since the above inequality is sharp, we can finally fix $\delta >0$ to be sufficiently small so that
\[
2-\frac{\theta}2 + \frac{\varepsilon}{4-2\delta } + \frac{3\delta  \theta }2 <2.
\]
This ensures that the power of $X$ on the right-hand side of \eqref{003} is smaller than $2$, which implies that 
\[
X(t)\lec_{\varepsilon_1, \varepsilon_2 , \delta , c_0, d} 1 
\]  for all $t\in (0,T)$, as required. \\

We note that boundedness of $X(t)$ implies regularity, which follows from the fact that $\| v \|_{W_2^{4,2}(\Omega^t )}+ \| \nabla p \|_{W_2^{2,1} (\Omega^t )} < \infty$,  see Theorem~\ref{thm_X_is_boss} below. In order to state it we first introduce constants which depend on the initial data and forcing,
\[
\begin{split}
D_1&\coloneqq \|f\|_{L_2(\Omega^t)}+\|v(0)\|_{L_2(\Omega)},\\
D_2&\coloneqq \|f_0\|_{L_1(0,t;L_\infty(\Omega))}+\|u(0)\|_{L_\infty(\Omega)},\\
D_3^2 &\coloneqq \frac1\nu \left( |\bar F_r |_{6/5,2,\Omega^t}^2 + |\bar F_\varphi |_{3,2,\Omega^t}^2  \right)+ |\Gamma (0) |_{2,\Omega }^2 +  |\Phi (0) |_{2,\Omega }^2,\\
D_4^2&\coloneqq (D_1^2D_2^2+|u_{,z}(0)|_{2,\Omega }^2+|f_0|_{2,\Omega^T}^2)/\nu,\\
D_5^2&\coloneqq D_1^2 \left(1+\frac{D_2^2}{\nu} \right)+|u_{,r}(0)|_{2,\Omega}^2+\frac{|f_0|_{2,\Omega^t}^2}{\nu },\\
D_6^2&\coloneqq  {D_2^2\over \nu \min\{1,D_2^2\}}\left( \|\bar F_\varphi\|^2_{L_2(0,t;L_{6/5}(\Omega))}+\|\bar F_r\|^2_{L_2(0,t;L_{6/5}(\Omega))}+\|\Gamma(0)\|_{L_2(\Omega)}^2\right)+ {1\over\min\{1,D_2^2\}}\|\Phi(0)\|^2_{L_2(\Omega)},\\
D_7^2&\coloneqq \|F_r\|_{L_2(0,t;L_{6/5}(\Omega))}^2+\|F_z\|_{L_2(0,t;L_{6/5}(\Omega))}^2+ \|\omega_r(0)\|_{L_2(\Omega)}^2+\|\omega_z(0)\|_{L_2(\Omega)}^2+D_1^2,\\
D_8^2&\coloneqq D_7^2+(D_4+D_5)\|f_\varphi\|_{L_2(S_1^t)}^2,\\
D_9&\coloneqq\|f_\varphi\|_{L_s(0,t;L_{3s\over 2s+1}(\Omega))}^s+ \|v_\varphi(0)\|_{L_s(\Omega)}^s,\\
D_{10}^2 &\coloneqq D_2 \left| \frac{f_{\varphi}}{r} \right|_{\infty,\Omega^t}+\frac{D_1^2}{\nu} + |v_\varphi (0) |_{\infty,\Omega}^2,\\
D_{11} &\coloneqq \left( \frac{4-2\delta }{c_0^{s-4+2\delta }} \frac{D_2^{4-2\delta }}{4\nu} \frac{R^{2\delta }}{\delta^2} D_1   \right)^{\frac{1}{4-2\delta }},\\
D_{12}& \coloneqq \left( \frac{4-2\delta }{c_0^{s-4+2\delta }} | f_\varphi |_{10/(1+6\delta ), \Omega^t } D_1^{3-2\delta } + | v_\varphi (0) |_{s,\Omega }^{4-2\delta }   \right)^{\frac{1}{4-2\delta }}.
\end{split}
\]
Note that $D_{11}$ and $D_{12}$ involve also a free parameter $\delta>0$, and $c_0 >0$ from \eqref{4.29}. In particular,  $D_{11},D_{12}$ are finite on $(0,T)$ if and only if \eqref{4.29} holds on $(0,T)$, and then \eqref{X_bd} can be stated more precisely as
\[
X(t) \leq \phi ( D_1, \ldots , D_{12} ) \qquad \text{ for }t\in (0,T)
\]
for some function $\phi$.

\begin{theorem}\label{thm_X_is_boss}
Suppose that $v$ is a solution of the Navier-Stokes equations on $(0,T)$ such that $D_1,\ldots , D_{12} <\infty$ for some $\delta >0$, $d\geq 6$, $c_0>0$. Then 
\[
\| v\|_{W_2^{4,2}(\Omega^T)}+\| \nabla p \|_{W_2^{2,1}(\Omega^T)} \leq C(D_1,\ldots ,D_{12}).
\]
Moreover, for each $l\geq 0$, $p\in (1,\infty )$, if $f\in W_p^{l,l/2} (\Omega^T)$ and $v(0)\in W^{l+2-2/p}_p (\Omega )$, then also
\[
\| v\|_{W_p^{l+2,\frac{l+2}2}(\Omega^T)} \leq C\left( p,l,D_1,\ldots ,D_{12},\| f \|_{W_{p}^{l,l/2} (\Omega^T)} \right).
\]
\end{theorem}
\begin{proof}
See Section~\ref{sec_reg}.
\end{proof}

We note that a result of Lei, Zhang \cite{LZ} (see also \cite{W}) shows that the axisymmetric Navier-Stokes equations are almost regular in the sense that a logarithmic modulus of continuity of $rv_{\varphi}$ at the axis is sufficient for regularity. The above proof of Theorem~\ref{T1} demonstrates a similar property from a different perspective.

\begin{remark}[Borderline lack of control of regularity of \eqref{nse}]\label{rem_crit}

We note that the above proof of Theorem~\ref{t1.3} demonstrates that the  axisymmetric Navier-Stokes equations are ``almost regular''.
In order to illustrate it, let us attempt to estimate $X(t)$ without making use of the assumption \eqref{4.29} of Theorem~\ref{t1.3}. In such case we can only use the order-reduction estimate \eqref{est2} on $|v_\varphi |_{\infty,\Omega^t}$. Then, instead of \eqref{I_control}, one can bound $I$ directly using the maximum principle~\eqref{2.9},
\[
I \leq D_2^{1-\varepsilon} |v_\varphi |^\varepsilon_{\infty, \Omega^t} \left| \frac{\Phi }{r^{1-\varepsilon_1}} \right|_{2,\Omega^t} \left| \frac{\Gamma }{r^{1-\varepsilon_2}} \right|_{2,\Omega^t} \lec \frac{D_2^{1-\varepsilon}}{\varepsilon_1 \varepsilon_2} |v_\varphi |_{\infty, \Omega^t}^\varepsilon X(t)^2 ,
\]
where $\varepsilon=\varepsilon_1+\varepsilon_2$ and we used the Hardy inequality \eqref{2.13} in the last inequality. Hence, applying this and \eqref{est2} in \eqref{005} gives
\[
X(t)^2 \lec_{\varepsilon_1,\varepsilon_2, \delta } 1+X(t)^{2+\frac{3(\varepsilon + \delta )}2}.
\]
We observe from this that, taking small $\varepsilon_1,\varepsilon_2, \delta >0$, we can obtain an ``almost linear'' estimate on $X$. 

Therefore, the above proof of Theorem~\ref{t1.3}, shows that one way of  obtaining a power on the right-hand side that is smaller than $2$ (and hence deducing regularity), is to use a more powerful order-reduction estimate \eqref{est3} (instead of~\eqref{est2}), in which case one needs to ``pay the price'' of assuming \eqref{4.29}.
\end{remark}

\begin{remark}[A comment about \eqref{4.29}]
Let us consider the expression in \eqref{4.29} for an arbitrary bounded function $f$,
\[
\frac{|f |_{d,\infty,\Omega^t}}{|f |_{\infty , \Omega^t}}.
\]
We first note that this expression is bounded above (for every $d\in [1,\infty )$) by $|\Omega |^{1/d}$, due to H\"older's inequality. Moreover, for every bounded function $f$ we have that
\eqnb\label{f_prop}
\lim_{d\to \infty } \frac{|f |_{d,\infty,\Omega^t}}{|f |_{\infty , \Omega^t}} =1.
\eqne
We emphasize that \eqref{f_prop} is a property of any $f\in L^\infty (\Omega^t)$, i.e. it has nothing to do with the Navier-Stokes equations.\\

\noindent Let us now consider \eqref{4.29} for $f\coloneqq v_\varphi$, where $v$ is a strong solution to the Navier-Stokes equations \eqref{nse_cylin} on a time interval $(0,T)$. Since $v_\varphi (t)\in L^\infty$ for every $t\in (0,T)$, the above observation implies that, for each $t\in (0,T)$, there exists $d(t)\in (6,\infty)$ such that 
\[
\frac{|v_\varphi |_{d,\infty,\Omega^t}}{|v_\varphi |_{\infty , \Omega^t}}  \geq c_d (t),
\]
for some $c_d (t)\in (0,1)$.\\

\noindent Suppose that $T\in (0,\infty )$ is the maximal time of existence of $v$. Then  $X(t)\to \infty $ as $ t\to T^-$ (see Section~\ref{sec_reg}). Consequently, Theorem~\ref{t1.3} implies that either
\begin{enumerate}
\item[(a)] $d(t)\to \infty$ as $t\to T^-$, or
\item[(b)] (if $d(t)$ remains bounded as $t\to T^-$, then) the corresponding lower bound $c_d(t)\to 0^+$ as $t\to T^-$.\vspace{0.5cm}\\
\end{enumerate} 
\end{remark}

The structure of the paper is as follows. After introducing some preliminary concepts in Section~\ref{s2} we discuss estimates on the modified stream function $\psi_1$  in Section~\ref{sec_psi1} and then discuss the main energy estimate \eqref{005} on $\Phi$, $\Gamma$, and \eqref{I_control}, in Section~\ref{sec_energy}. We then provide energy estimates for the swirl $u$ in Section~\ref{sec_swirl}, and show the order reduction estimates  (Theorem~\ref{T1})  in Section~\ref{sec_order_red}. In the last Section~\ref{sec_reg} we show that the regularity persists as long as $X(t)$ remains bounded, as mentioned below \eqref{2.23}.

\section{Preliminaries}\label{s2}

\subsection{Notation}
We will use the notation
\[
\intop_\Omega f  \coloneqq \intop_\Omega f \d x, \qquad \intop_{\Omega^t} f  \coloneqq \intop_{\Omega^t} f \d x\, \d t', \qquad 
\]
We often use cylindrical coordinates in integration, in which case we always write $\d r \d z$, in order to keep track of the Jacobian $r$.\\
We will use the following notation for Lebesque spaces
\eqnb\label{def_Lp}
\begin{split}
|u|_{p,\Omega}&\coloneqq \|u\|_{L_p(\Omega)},\hspace{2cm} |u|_{p,\Omega^t} \coloneqq \|u\|_{L_p(\Omega^t)},\\
|u|_{p,q,\Omega^t} &\coloneqq \|u\|_{L_q(0,t;L_p(\Omega))},
\end{split}
\eqne
where $p,q\in[1,\infty]$. We use standard definition of Sobolev spaces $W_p^{s} (\Omega )$, and we set $H^s (\Omega )\coloneqq W_2^{s} (\Omega )$, and 
\[
\begin{split}
\|u\|_{s,\Omega} &\coloneqq \|u\|_{H^s(\Omega)},\hspace{2.05cm} \|u\|_{s,p,\Omega}\coloneqq  \|u\|_{W_p^{s}(\Omega)}, \\
\|u\|_{k,p,q,\Omega^t}&\coloneqq \|u\|_{L_q(0,t;W_p^k(\Omega))},\hspace{1cm} \|u\|_{k,p,\Omega^t}\coloneqq \|u\|_{k,p,p,\Omega^t}.
\end{split}
\]
Finally, we  note that
\eqnb\label{axisym_grad}
|\nabla_{x_1,x_2} w|^2 = |w_{r,r}|^2 +|w_{\varphi ,r}|^2 +  |w_{z,r}|^2 +\frac{w_r^2}{r^2} +\frac{w_\varphi^2}{r^2} 
\eqne
for any axisymmetric vector field $w$
\subsection{Inequalities}

\begin{lemma}[Hardy inequality, see Lemma 2.16 in \cite{BIN}]\label{l2.6}
Let $p \in [1,\infty ]$, $\beta\not=1/p$, and let $F(x)\coloneqq \intop_0^xf(y)\d y$ for $\beta>1/p$ and $F(x)\coloneqq \intop_x^\infty f(y)\d y$ for $\beta<1/p$. Then
\begin{equation}
|x^{-\beta}F|_{p,\R_+}\le{1\over|\beta-{1\over p}|}|x^{-\beta+1}f|_{p,\R_+}.
\label{2.13}
\end{equation}
\end{lemma}
\begin{lemma}[Sobolev interpolation, see Sect. 15 in \cite{BIN}]\label{l2.8}
Let $\theta$ satisfy the equality
\begin{equation}
{n\over p}-r=(1-\theta){n\over p_1}+\theta\bigg({n\over p_2}-l\bigg)\quad {r\over l}\le\theta\le 1,
\label{2.16}
\end{equation}
where $1\le p_1\le\infty$, $1\le p_2\le\infty$, $0\le r<l$.\\
Then the interpolation holds
\begin{equation}
\sum_{|\alpha|=r}|D^\alpha f|_{p,\Omega}\le c|f|_{p_1,\Omega}^{1-\theta}\|f\|_{W_{p_2}^l(\Omega)}^\theta,
\label{2.17}
\end{equation}
where $\Omega\subset\R^n$ and $D^\alpha f=\partial_{x_1}^{\alpha_1}\dots\partial_{x_n}^{\alpha_n}f$, $|\alpha|=\alpha_1+\alpha_2+\cdots+\alpha_n$.
\end{lemma}
\begin{lemma}[Hardy interpolation, see Lemma~2.4 in \cite{CFZ}]\label{l2.9}
Let $f\in C^\infty((0,R)\times(-a,a))$, $f|_{r\ge R}=0$. Let $1<p< 3$, $0\le s\le p$, $s< 2$, $q\in\big[p,{p(3-s)\over 3-p}\big]$. Then there exists a positive constant $c=c(p,s)$ such that
\begin{equation}
\left(\int_\Omega{|f|^q\over r^s}\right)^{1/q}\le c|f|_{p,\Omega}^{{3-s\over q}-{3\over p}+1} |\nabla f|_{p,\Omega}^{{3\over p}-{3-s\over q}},
\label{2.18}
\end{equation}
where $f$ does not depend on $\varphi$.
\end{lemma}

\subsection{Basic estimates}
We first recall the energy for any regular solution $v$ to \eqref{nse}--\eqref{bcs}  
\eqnb\label{2.1}
|v(t)|_{2,\Omega}^2+\nu\intop_{\Omega^t}\left( |\nabla v_r|^2+|\nabla v_\varphi|^2+|\nabla v_z|^2 +{v_r^2\over r^2}+{v_\varphi^2\over r^2}\right)\lec  D_1^2,
\eqne
which follows by multiplying the Navier-Stokes equation \eqref{nse} by $v$, integrating by parts, and recalling \eqref{axisym_grad}.

As for the swirl $u=rv_\varphi$, we have the following.
\begin{lemma}[Maximum principle for the swirl]\label{l2.4}
For any regular solution $v$ to \eqref{nse}--\eqref{bcs}  we have 
\begin{equation}
|u(t)|_{\infty,\Omega}\le D_2.
\label{2.9}
\end{equation}
\end{lemma}
\begin{proof}
Multiplying the swirl equation \eqref{u_eq}) by $u|u|^{s-2}$,  $s>2$, integrating over $\Omega$ and by parts, we obtain
\[
{1\over s}{\d\over \d t}|u|_{s,\Omega}^s+{4\nu(s-1)\over s^2}|\nabla|u|^{s/2}|_{2,\Omega}^2+{\nu\over s}\intop_\Omega(|u|^s)_{,r}\d r\d z=\intop_\Omega f_0u|u|^{s-2}.
\]
Noting that  $u|_{r=0}=u|_{r=R}=0$ (by \eqref{2.20} and \eqref{bcs}), we see that the last term on the left-hand side  vanishes, and so
\[
{\d\over \d t}|u|_{s,\Omega}\le|f_0|_{s,\Omega}.
\]
Integration in time and taking  $s\to\infty$ gives \eqref{2.9}.
\end{proof}

\begin{lemma}[Energy estimates for $\psi$ and $\psi_1$]\label{l2.7} For every regular solution $v$ to \eqref{nse}--\eqref{bcs}  
\begin{align}
\|\psi\|_{1,\Omega}^2+|\psi_1|_{2,\Omega}^2&\lec D_1^2,
\label{2.14}\\
\|\psi_{,z}\|_{1,2,\Omega^t}^2+|\psi_{1,z}|_{2,\Omega^t}^2&\lec D_1^2.
\label{2.15}
\end{align}
\end{lemma}

\begin{proof}
Multiplying $(\ref{1.14})_1$ by $\psi$, and integrating over $\Omega$ we obtain
\[\begin{split}
|\nabla\psi|_{2,\Omega}^2+|\psi_1|_{2,\Omega}^2&=\intop_\Omega\omega_\varphi\psi =\intop_\Omega(v_{r,z}-v_{z,r})\psi =\intop_\Omega(v_z\psi_{,r}-v_r\psi_{,z})\\
&\le (|\psi_{,r}|_{2,\Omega}^2+ |\psi_{,z}|_{2,\Omega}^2)/2 + c (|v_r|_{2,\Omega}^2+|v_z|_{2,\Omega}^2),\end{split}
\]
where we integrated by parts and used the boundary condition $\psi|_S=0$ (recall \eqref{1.14}) in the third equality. For \eqref{2.15} we differentiate $(\ref{1.14})_1$ with respect to $z$, multiply by $\psi_{,z}$ and integrate over $\Omega^t$ to obtain
\[\begin{split}
\intop_{\Omega^t}|\nabla\psi_{,z}|^2+\intop_{\Omega^t}|\psi_{1,z}|^2 = \intop_{\Omega^t}\omega_{\varphi,z}\psi_{,z}=-\intop_{\Omega^t}\omega_\varphi\psi_{,zz} \le|\psi_{,zz}|_{2,\Omega^t}^2/2+ c |\omega_\varphi|_{2,\Omega^t}^2,\\
\end{split}
\]
as required, where we used boundary condition $\omega_\varphi|_S=0$ (recall \eqref{bcs}) in the second equality. 
\end{proof}

\section{Estimates for the modified stream function $\psi_1$}\label{sec_psi1}

Here we introduce some estimates of $\psi_1$ in terms of $\Gamma$.

\subsection{Weighted Sobolev estimates for $\psi_1$}
\begin{lemma}[see Lemma 4.2 \cite{NZ}]\label{l2.18}
If $\psi_1$ is a sufficiently regular solution to \eqref{psi1_eq}, then 
\begin{equation}
\intop_\Omega(\psi_{1,zzz}^2+\psi_{1,rzz}^2)r^{2\mu}+2\mu(1-\mu)\intop_\Omega\psi_{1,zz}^2r^{2\mu-2}\le c\intop_\Omega\Gamma_{,z}^2r^{2\mu}.
\label{2.26}
\end{equation}
\end{lemma}

\begin{proof}
We differentiate (\ref{psi1_eq}) with respect to $z$, multiply by $-\psi_{1,zzz}r^{2\mu}$ and integrate over $\Omega$ to obtain
\begin{equation}\eqal{
&\intop_\Omega\psi_{1,rrz}\psi_{1,zzz}r^{2\mu}+\intop_\Omega\psi_{1,zzz}^2r^{2\mu}\cr
&\quad+3\intop_\Omega{1\over r}\psi_{1,rz}\psi_{1,zzz}r^{2\mu }=-\intop_\Omega\Gamma_{,z}\psi_{1,zzz}r^{2\mu}.\cr}
\label{2.27}
\end{equation}
In view of \eqref{bc_psi1,zz}, the first integral on the left-hand side of (\ref{2.27}) equals
\begin{equation}\eqal{
-&\intop_\Omega\psi_{1,rrzz}\psi_{1,zz}r^{2\mu}=-\intop_\Omega(\psi_{1,rzz}\psi_{1,zz}r^{2\mu+1})_{,r}\d r\d z\cr
&\quad+\intop_\Omega\psi_{1,rzz}^2r^{2\mu}+(2\mu+1)\intop_\Omega\psi_{1,rzz}\psi_{1,zz}r^{2\mu}\d r\d z.\cr}
\label{2.29}
\end{equation}
Since $\psi_1|_{r=R}=\psi_{1,r}|_{r=0}=0$ (by \eqref{psi1_eq} and \eqref{2.23}), the first term on the right-hand side of  \eqref{2.29} vanishes. Integrating by parts with respect to $z$ in the last term on the left-hand side of (\ref{2.27}) and using \eqref{bc_psi1,zz}, it takes the form
\begin{equation}
-3\intop_\Omega\psi_{1,rzz}\psi_{1,zz}r^{2\mu} \,\d r\d z.
\label{2.30}
\end{equation}
Using (\ref{2.29}) and (\ref{2.30}) in (\ref{2.27}) yields
\eqnb\label{2.31}
\intop_\Omega(\psi_{1,zzz}^2+\psi_{1,rzz}^2)r^{2\mu}+2(\mu-1)\intop_\Omega\psi_{1,rzz}\psi_{1,zz}r^{2\mu}\d r\d z=-\intop_\Omega\Gamma_{,z}\psi_{1,zzz}r^{2\mu}.
\eqne
The second term on the left-hand side of (\ref{2.31}) equals
\begin{equation}(\mu-1)\intop_\Omega\partial_r(\psi_{1,zz}^2)r^{2\mu}\d r\d z=(\mu-1)\intop_\Omega\partial_r(\psi_{1,zz}^2r^{2\mu})\d r\d z+2\mu(1-\mu)\intop_\Omega\psi_{1,zz}^2r^{2\mu-1}\d r\d z,
\label{2.32}
\end{equation}
where the first integral vanishes because $\psi_{1,zz}|_{r=R}=0$ (recall \eqref{psi1_eq}) and $\psi_{1,zz}^2r^{2\mu}|_{r=0}=0$ (recall \eqref{2.22}). Using (\ref{2.32}) in (\ref{2.31}) and applying the H\"older and Young inequalities to the r.h.s. of (\ref{2.31}), we obtain (\ref{2.26}), as required.
\end{proof}

\subsection{Elliptic estimates for the modified stream function $\psi_1$}\label{s3}

We recall that the modified stream function $\psi_1$ is a solution to the problem \eqref{psi1_eq},
\[
-\Delta\psi_1-{2\over r}\psi_{1,r}=\Gamma,\qquad \left. \psi_1 \right|_S=0.
\]
In this section we prove $H^2$ and $H^3$ elliptic estimates for $\psi_1$, in cylindrical coordinates.

\begin{lemma}[$H^2$ elliptic estimate on $\psi_1$, see Lemma 3.1 in \cite{Z1}]\label{l3.1}
If $\psi_1$ is a sufficiently regular solution to \eqref{psi1_eq} then 
\eqnb\label{psi1_h2}
\intop_\Omega \left( \psi_{1,rr}^2+\psi_{1,rz}^2+\psi_{1,zz}^2 + \frac{\psi_{1,r}^2}{r^2} \right) +\int_{-a}^a \left( \left.\psi_{1,z}^2\right|_{r=0} +  \left.\psi_{1,r}^2\right|_{r=R} \right) \d z \le c|\Gamma|_{2,\Omega}^2.
\eqne
\end{lemma}

\begin{proof}
We multiply \eqref{psi1_eq} by $\psi_{1,zz}$ and integrate over $\Omega$ to obtain
\begin{equation}
-\intop_\Omega \left( \psi_{1,rr}\psi_{1,zz}+ \psi_{1,zz}^2 + 3 \frac{ \psi_{1,r}}r \psi_{1,zz}\right) =\intop_\Omega\Gamma\psi_{1,zz}.
\label{3.3}
\end{equation}
Integrating by parts with respect to $r$ in the first term gives
\[
\begin{split} 
-\intop_\Omega(\psi_{1,r}\psi_{1,zz}r)_{,r}\d r\d z+\int_\Omega\psi_{1,r}\psi_{1,zzr} &+ \int_\Omega\psi_{1,r}\psi_{1,zz}\d r\d z\\
&-\intop_\Omega  \psi_{1,zz}^2  - 3\intop_\Omega\psi_{1,r}\psi_{1,zz}\d r\d z=\intop_\Omega\Gamma \psi_{1,zz}.
\end{split}\]
Thus
\eqnb\label{3.4}
\begin{split}
-\int_{-a}^a \left[ \psi_{1,r}\psi_{1,zz}r\right]_{r=0}^{r=R}\d z+\intop_\Omega\psi_{1,r}\psi_{1,zzr}-\int_\Omega\psi_{1,zz}^2 -2\intop_\Omega\psi_{1,r}\psi_{1,zz}\d r\d z&= \intop_\Omega\Gamma\psi_{1,zz}.
\end{split}
\eqne
We note that the the first integral vanishes since $\psi_{1,r}|_{r=0}=0$ (recall expansion \eqref{2.23}) and $\psi_{1,zz}|_{r=R}=0$. We now integrate by parts with respect to $z$ in the second and the last terms on the left-hand side and use that $\psi_{1,r}|_{S_2}=0$ (since $\psi_1|_{S}=0$, recall \eqref{psi1_eq}), and we multiply by $-1$, to   obtain 
\eqnb\label{3.7}
\intop_\Omega(\psi_{1,zr}^2+\psi_{1,zz}^2) -2\intop_\Omega\psi_{1,rz}\psi_{1,z}\d r\d z =-\intop_\Omega\Gamma \psi_{1,zz}.
\eqne
We note that the last term on the left-hand side equals
\[
-\intop_\Omega (\psi_{1,z}^2)_{,r} \d r \d z = -\int_{-a}^a\left[ \psi_{1,z}^2 \right]_{r=0}^{r=R}\d z=\intop_{-a}^a\left. \psi_{1,z}^2\right|_{r=0}\d z,
\]
since $\psi_{1,z}|_{r=R}=0$. Applying this in \eqref{3.7}, and using the Young inequality to absorb $\psi_{1,zz}$ by the left-hand side, we obtain 
\begin{equation}
\intop_\Omega \left( \psi_{1,rz}^2+\psi_{1,zz}^2\right) +\intop_{-a}^a\left. \psi_{1,z}^2\right|_{r=0}\d z\le c|\Gamma|_{2,\Omega}^2.
\label{3.8}
\end{equation}

We now multiply $\eqref{psi1_eq}_1$ by $\psi_{1,r}/r$ and integrate over $\Omega$ to obtain
\eqnb\label{3.9} 3\intop_\Omega \frac{\psi_{1,r}^2}{r^2} = -\intop_\Omega\left( \psi_{1,rr}\frac{\psi_{1,r}}r  + \psi_{1,zz}\frac{\psi_{1,r}}r + \Gamma\frac{ \psi_{1,r}}r \right) .
\eqne
The first term on the right-hand side equals 
\[
-{1\over 2}\intop_\Omega\partial_r\psi_{1,r}^2\d r\d z=-{1\over 2}\intop_{-a}^a \left[ \psi_{1,r}^2\right]_{r=0}^{r=R}\d z=-{1\over 2}\intop_{-a}^a\left. \psi_{1,r}^2\right|_{r=R}\d z.
\]
where we used that $\psi_{1,r}|_{r=0}=0$ (recall expansion \eqref{2.23}) in the last equality. As for the other terms on the right-hand side of \eqref{3.9} we apply Young's inequality to absorb $\psi_{1,r}/r$ by the left-hand side. We obtain
\[\intop_\Omega\frac{\psi_{1,r}^2}{r^2} +{1\over 2}\intop_{-a}^a \left. \psi_{1,r}^2\right|_{r=R}\d z\lec |\psi_{1,zz}|_{2,\Omega}^2+|\Gamma |_{2,\Omega}^2.
\]
The claim \eqref{psi1_h2} follows from this, \eqref{3.8}, and from the equation $\eqref{psi1_eq}_1$ for $\psi_1$, which lets us estimate $\psi_{1,rr}$ in terms of $\psi_{1,zz}$, $\psi_{1,r}/r$.
\end{proof}

\begin{lemma}[$H^3$ elliptic estimates on $\psi_1$]\label{l3.2}
If $\psi_1$ is a sufficiently regular solution to \eqref{psi1_eq} then 
\begin{equation}
\intop_\Omega(\psi_{1,zzr}^2+\psi_{1,zzz}^2)+\intop_{-a}^a\psi_{1,zz}^2\bigg|_{r=0}\d z\le c|\Gamma_{,z}|_{2,\Omega}^2
\label{3.13}
\end{equation}
and
\begin{equation}\eqal{
&\intop_\Omega(\psi_{1,rrz}^2+\psi_{1,rzz}^2+\psi_{1,zzz}^2)+ \intop_{-a}^a\psi_{1,zz}^2\bigg|_{r=0}\d z+\intop_{-a}^a \psi_{1,rz}^2\bigg|_{r=R}\d z\cr
&\le c|\Gamma_{,z}|_{2,\Omega}^2.\cr}
\label{3.14}
\end{equation}
as well as
\begin{equation}
\bigg|{1\over r}\psi_{1,rz}\bigg|_{2,\Omega}\le c|\Gamma_{,z}|_{2,\Omega}
\label{3.24}
\end{equation}
\end{lemma}

\begin{proof}
First we show (\ref{3.13}). We differentiate $\eqref{psi1_eq}_1$ with respect to $z$, multiply by $-\psi_{1,zzz}$ and integrate over $\Omega$ to obtain
\eqnb\label{3.15}
\begin{split}
\intop_\Omega\psi_{1,rrz}\psi_{1,zzz} +\intop_\Omega\psi_{1,zzz}^2&+3\intop_\Omega{1\over r}\psi_{1,rz}\psi_{1,zzz}=-\intop_\Omega\Gamma_{,z}\psi_{1,zzz}.
\end{split} 
\eqne
Integrating by parts with respect to $z$ in the first term yields
\eqnb \label{3.16}
\intop_\Omega\psi_{1,rrz}\psi_{1,zzz}= \intop_\Omega(\psi_{1,rrz}\psi_{1,zz})_{,z}- \intop_\Omega\psi_{1,rrzz}\psi_{1,zz},
\eqne
where the first term vanishes due to \eqref{bc_psi1,zz}. Integrating the last integral in  \eqref{3.16} by parts with respect to $r$  gives
$$
-\intop_\Omega(\psi_{1,rzz}\psi_{1,zz}r)_{,r}\d r\d z+\intop_\Omega\psi_{1,rzz}^2+ \intop_\Omega\psi_{1,rzz}\psi_{1,zz}\d r\d z,
$$
where the first integral vanishes, since $\psi_{1,rzz}|_{r=0}=\psi_{1,zz}|_{r=R}=0$ (recall \eqref{2.23} and \eqref{psi1_eq}).

Thus, (\ref{3.15}) becomes
\eqnb\label{3.17} 
\intop_\Omega(\psi_{1,rzz}^2+\psi_{1,zzz}^2)+\intop_\Omega\left( \psi_{1,rzz}\psi_{1,zz} + 3\psi_{1,rz}\psi_{1,zzz}\right) \d r\d z=-\intop_\Omega\Gamma_{,z}\psi_{1,zzz}.
\eqne
Integrating by parts with respect to $z$ in the last term on the left-hand side of \eqref{3.17} and using that $\psi_{1,zz}|_{S_2}=0$ (recall \eqref{bc_psi1,zz}) we get
\eqnb\label{3.18}
\intop_\Omega(\psi_{1,rzz}^2+\psi_{1,zzz}^2)-\intop_\Omega\partial_r\psi_{1,zz}^2\d r\d z= -\intop_\Omega\Gamma_{,z}\psi_{1,zzz}.
\eqne
Recalling \eqref{psi1_eq} that  $\psi_{1,zz}|_{r=R}=0$, and using Young's inequality to absorb $\psi_{1,zzz}$ we obtain 
\[
\intop_\Omega \left( \psi_{1,rzz}^2+\psi_{1,zzz}^2\right) +\intop_{-a}^a\left. \psi_{1,zz}^2\right|_{r=0}\d z\lec |\Gamma_{,z}|_{2,\Omega}^2.
\]
which gives \eqref{3.13}.

As for \eqref{3.14}, we differentiate $(\eqref{psi1_eq})_1$ with respect to $z$, multiply by $\psi_{1,rrz}$ and integrate over $\Omega$ to obtain
\eqnb\label{3.19}
-\intop_\Omega \left( \psi_{1,rrz}^2+ \psi_{1,zzz}\psi_{1,rrz}+3{1\over r}\psi_{1,rz}\psi_{1,rrz}\right) =\intop_\Omega\Gamma_{,z}\psi_{1,rrz}.
\eqne
We integrate the second term on the left-hand side by parts in $z$, and recall \eqref{bc_psi1,zz} that  $\psi_{1,zz}|_{S_2}=0$, to get
\[
\begin{split}
-\intop_\Omega\psi_{1,zzz}\psi_{1,rrz}&=\intop_\Omega\psi_{1,zz}\psi_{1,rrzz}\\
&=\intop_\Omega(\psi_{1,zz}\psi_{1,rzz}r)_{,r}\d r\d z-\intop_\Omega\psi_{1,rzz}^2- \intop_\Omega\psi_{1,zz}\psi_{1,rzz}\d r\d z.
\end{split}
\]
We note that the first term on the right-hand side vanishes since $\psi_{1,rzz}|_{r=0}=0$ (recall \eqref{2.23}) and $\psi_{1,zz}|_{r=R}=0$ (recall \eqref{psi1_eq}), and so \eqref{3.19} becomes 
\eqnb\label{3.21}
\intop_\Omega \left( \psi_{1,rrz}^2+\psi_{1,rzz}^2 \right)+\intop_\Omega \left( \psi_{1,zz}\psi_{1,rzz}+3 \psi_{1,rz}\psi_{1,rrz} \right) \d r\d z =-\intop_\Omega\Gamma_{,z}\psi_{1,rrz}.
\eqne
Since the second term above equals
\[
{1\over 2}\intop_{-a}^a\left[ \psi_{1,zz}^2\right]_{r=0}^{r=R}\d z=-{1\over 2}\intop_{-a}^a\left. \psi_{1,zz}^2\right|_{r=0}\d z
\]
(as $\psi_{1,zz}|_{r=R}=0$, recall \eqref{psi1_eq}), and the last term on the left-hand side of \eqref{3.21} equals 
\[
\frac32 \intop_\Omega\partial_r\psi_{1,rz}^2\d r\d z=\frac32 \intop_{-a}^a\left[ \psi_{1,rz}^2\right]_{r=0}^{r=R}\d z=\frac32 \intop_{-a}^a\left. \psi_{1,rz}^2\right|_{r=R}\d z
\]
(as $\psi_{1,rz}|_{r=0}=0$, recall \eqref{2.23}), \eqref{3.21} becomes 
\eqnb\label{3.22}
\intop_\Omega(\psi_{1,rrz}^2+\psi_{1,rzz}^2)+ \intop_{-a}^a \left( -\frac12 \left. \psi_{1,zz}^2\right|_{r=0} + \frac32 \left. \psi_{1,rz}^2\right|_{r=R} \right) \d z=-\intop_\Omega\Gamma_{,z}\psi_{1,rrz}.
\eqne
We now use Young's inequality to absorb $\psi_{1,rrz}$ by the left-hand side to obtain \eqref{3.14}, which in turn implies \eqref{3.24} by differentiating $\eqref{psi1_eq}_1$ in $z$. 
\end{proof}

\section{The energy estimates for $\Phi$ and $\Gamma$}\label{sec_energy}

Here we show the energy estimate \eqref{005} for $\Gamma$ and $\Phi$, and we also verify the claimed control \eqref{I_control} of $I= |\int_{\Omega^t } {v_\varphi } \Phi \Gamma/r| $.

\begin{lemma}[Energy estimate for $\Phi$, $\Gamma$]\label{l4.1}
If $v$ is a regular solution of \eqref{nse}--\eqref{bcs}  on $(0,T)$ then, for every $t\in (0,T)$,
\eqnb\label{4.1}
D_2^2 \| \Gamma \|^2_{V(\Omega^t)} +\| \Phi \|^2_{V(\Omega^t)}  \lec D_2^2 \left(1+\frac{|v_\varphi|_{\infty,\Omega^t }^{\delta }R^{\delta }}{\delta   D_2^{\delta }} \right) \left(I + D_3^2 \right),
\eqne
and moreover, if $v_\varphi\in L^\infty (0,t;L^d (\Omega))$ for some $d>3$ and $\varepsilon_1, \varepsilon_2>0$ are sufficiently small such that
\[
\theta\coloneqq \left(1-{3\over d}\right)\varepsilon_1-{3\over d}\varepsilon_2\in (0,1), \quad 1+\frac{\varepsilon_2}{\varepsilon_1} < \frac{d}3,
\]
then
\eqnb\label{4.12}
I \lec D_2^{1-\varepsilon}|v_\varphi|_{d,\infty,\Omega^t}^\varepsilon{R^{\varepsilon_2}\over\varepsilon_2} |\Phi|_{2,\Omega^t}^\theta\| \Phi \|_{V(\Omega^t )}^{1-\theta}\| \Gamma \|_{V(\Omega^t )},
\eqne
where $\varepsilon\coloneqq \varepsilon_1+\varepsilon_2$.  
\end{lemma}
(Recall \eqref{def_V} that $\|w\|_{V(\Omega^t)}\coloneqq |w|_{2,\infty,\Omega^t}+|\nabla w|_{2,\Omega^t}$.)
\begin{proof}
We multiply \eqref{1.17} by $\Phi$ and integrate over $\Omega$ to obtain
\eqnb\label{4.2}
{1\over 2}{\d\over \d t}|\Phi|_{2,\Omega}^2+\nu |\nabla\Phi|_{2,\Omega}^2-\nu \intop_{-a}^a\Phi^2\bigg|_{r=0}^{r=R}\d z =\intop_\Omega(\omega_r\partial_r+\omega_z\partial_z){v_r\over r}\Phi +\intop_\Omega\bar F_r\Phi,
\eqne
where the last term on the left-hand side equals $\intop_{-a}^a\Phi^2|_{r=0}\d z$, due to \eqref{bcs1}.  Recalling \eqref{1.13} that $\omega_r=-v_{\varphi ,z}$, $\omega_z=(rv_{\varphi })_{,r}/r$, we can integrate in the first term on the right-hand side by parts,  
\[
\begin{split}
\intop_\Omega(\omega_r\partial_r+\omega_z\partial_z){v_r\over r}\Phi &=\intop_\Omega\left( -v_{\varphi,z}\left( \frac{v_r}{r} \right)_{,r}+\frac{(rv_\varphi)_{,r}}{r^2}v_{r,z} \right) \Phi \, r\,\d r\,\d z\\
&=\intop_\Omega v_\varphi\left(\left(\frac{v_r}{r}\right)_{,rz}\Phi+\left( \frac{v_r}{r} \right)_{,r} \Phi_{,z} \right)-\intop_\Omega v_\varphi\bigg(\left(\frac{v_r}{r}\right)_{,rz}\Phi+\left( \frac{v_r}{r} \right)_{,z} \Phi_{,r}\bigg)\\
&=-\intop_\Omega v_\varphi\left( \psi_{1,zr} \Phi_{,z}-\psi_{1,zz} \Phi_{,r}\right)\\
&\leq \intop_{\Omega} \left| rv_{\varphi } \frac{\psi_{1,rz} }r \Phi_{,z} \right|  + \intop_{\Omega} \left| r^{1-\delta } v_{\varphi } \frac{\psi_{1,zz} }{r^{1-\delta }} \Phi_{,r} \right|  \\
&\leq |rv_\varphi|_{\infty,\Omega}\bigg|{\psi_{1,rz}\over r}\bigg|_{2,\Omega}|\Phi_{,z}|_{2,\Omega}+|r^{1-\delta } v_\varphi|_{\infty,\Omega}\left|\frac{\psi_{1,zz}}{r^{1-\delta }}\right|_{2,\Omega}|\Phi_{,r}|_{2,\Omega} \\
&\lec  D_2|\nabla \Phi |_{2,\Omega}\left( |\Gamma_{,z}|_{2,\Omega} +\frac{|v_\varphi|_{\infty,\Omega}^{\delta }}{\delta  D_2^{\delta }} |\psi_{1,zzr}r^{\delta }|_{2,\Omega} \right) \\
&\leq D_2|\nabla \Phi |_{2,\Omega}|\nabla \Gamma |_{2,\Omega} \left(1+\frac{|v_\varphi|_{\infty,\Omega}^{\delta }R^{\delta }}{\delta   D_2^{\delta }} \right),
\end{split}
\]
where, in the second line, the boundary term on $S_2$ vanishes because  $v_{\varphi,z}/r|_{S_2}=\Phi|_{S_2}=0$ (recall \eqref{bcs} and \eqref{bcs1}) and the second boundary term equals
\[
\intop_{-a}^a\left[ rv_\varphi\partial_z{v_r\over r}\Phi\right]_{r=0}^{r=R}\d z=0
\]
(since $v_\varphi|_{r=R}=0$ by \eqref{bcs}  $rv_\varphi\partial_z{v_r\over r}\Phi|_{r=0}=0$ by \eqref{2.19}--\eqref{2.20}). We also used the maximum principle \eqref{2.9}, \eqref{3.24} and the Hardy inequality \eqref{2.13} in the third inequality and \eqref{3.13} in the fourth inequality. Applying this in \eqref{4.2}  yields
\begin{equation}
\frac12 {\d\over \d t}|\Phi|_{2,\Omega}^2+\frac{\nu }2 |\nabla\Phi|_{2,\Omega}^2\lec  \frac{D_2^2}{2\nu }|\nabla \Gamma |_{2,\Omega}^2 \left(1+\frac{|v_\varphi|_{\infty,\Omega}^{2\delta }R^{2\delta }}{\delta ^2  D_2^{2\delta }} \right)+\frac{1}{\nu }|\bar F_r|_{6/5,\Omega}^2.
\label{4.6}
\end{equation}
Multiplying \eqref{1.18} by $\Gamma$, integrating over $\Omega$ and using the boundary conditions we have
\[ {1\over 2}{\d\over \d t}|\Gamma|_{2,\Omega}^2+\nu |\nabla\Gamma|_{2,\Omega}^2-\nu \intop_{-a}^a\Gamma^2\bigg|_{r=0}^{r=R}\d z=-2 \int_\Omega \frac{v_\varphi}r \Phi \Gamma +\intop_\Omega\bar F_\varphi\Gamma . \]
Applying the H\"older and Young inequality to the last term, and then multiplying the resulting inequality by $cD_2^2 (1+ |v_\varphi|_{\infty,\Omega}^{2\delta }R^{2\delta }/({\delta ^2  D_2^{2\delta }}))$ and adding  to \eqref{4.6} gives
\[\begin{split}
D_2^2{\d\over \d t}|\Gamma|_{2,\Omega}^2+\nu D_2^2|\nabla \Gamma |_{2,\Omega}^2&+{\d\over \d t}|\Phi|_{2,\Omega}^2+\nu |\nabla\Phi|_{2,\Omega}^2\\
&\lec D_2^2\left(1+\frac{|v_\varphi|_{\infty,\Omega}^{\delta }R^{\delta }}{\delta   D_2^{\delta }} \right) \left(\int_\Omega \frac{v_\varphi}r \Phi \Gamma + \frac{1}{2 \nu }\left( |\bar F_r|_{6/5,\Omega}^2 + |\bar F_\varphi |_{6/5,\Omega}^2 \right) \right). 
\end{split}\]
Integration in time gives \eqref{4.1}, as desired. \\

\noindent As for \eqref{4.12}, we note that
\eqnb\label{004}
\begin{split}
I&\le \intop_{\Omega^t}|rv_\varphi|^{1-\varepsilon}|v_\varphi|^\varepsilon\bigg|{\Phi\over r^{1-\varepsilon_1}}\bigg|\,\bigg|{\Gamma\over r^{1-\varepsilon_2}}\bigg| \\
&\le D_2^{1-\varepsilon}\left(\intop_{\Omega^t}|v_\varphi|^{2\varepsilon}\left|{\Phi\over r^{1-\varepsilon_1}}\right|^2 \right)^{1/2}\left|{\Gamma\over r^{1-\varepsilon_2}}\right|_{2,\Omega^t}\\
&\lec \frac{D_2^{1-\varepsilon}R^{\varepsilon_2}}{\varepsilon_2} | v_{\varphi } |^\varepsilon_{d,\infty,\Omega^t} \left|{\Phi\over r^{1-\varepsilon_1}}\right|_{\frac{2d}{d-2\varepsilon },2,\Omega^t} \left|\nabla \Gamma \right|_{2,\Omega^t}\\
&\lec \frac{D_2^{1-\varepsilon}R^{\varepsilon_2}}{\varepsilon_2} | v_{\varphi } |^\varepsilon_{d,\infty,\Omega^t} \left| \Phi \right|^\theta_{2,\Omega^t}\left| \nabla \Phi \right|^{1-\theta}_{2,\Omega^t}  \left|\nabla \Gamma \right|_{2,\Omega^t}, 
\end{split}
\eqne
as required, where we used the Hardy inequality \eqref{2.13} in the third inequality, and the Hardy interpolation (Lemma~\ref{l2.9}) in the last inequality, where we also recalled that $\theta \coloneqq \left(1-{3\over d}\right)\varepsilon_1-{3\over d}\varepsilon_2$. Note that $2d/(d-2\varepsilon ) \geq 2$ and (by assumption)
\[  \frac{2d}{d-2\varepsilon } \leq 2\left( 3 - \frac{2d(1-\varepsilon_1 )}{d-2\varepsilon } \right) \Leftrightarrow \frac{d}3\geq \frac{\varepsilon }{\varepsilon_1} ,\]
 as required by the assumption of the Hardy interpolation Lemma~\ref{l2.9}.
\end{proof}

\section{Estimates for swirl $u=rv_\varphi$}\label{sec_swirl}

Here we derive energy estimates for $\nabla u$. Recall that the swirl $u=rv_\varphi$ satisfies \eqref{u_eq}, namely
\[
u_{,t}+v\cdot\nabla u-\nu\Delta u+2\nu {u_{,r}\over r}= f_0,
\]
with boundary conditions $u=0$ on $S_1$ and $u_{,z}=0 $ on $S_2$. 

\begin{lemma}[see Lemma 5.1 in \cite{Z2}]\label{l5.1}
Any regular solution $u$ to \eqref{u_eq} satisfies 
\begin{equation}
\hspace{2cm}|u_{,z}(t)|_{2,\Omega}^2+\nu|\nabla u_{,z}|_{2,\Omega^t}^2\lec D_4^2,
\label{5.2}
\end{equation}
\eqnb\label{5.3}
|u_{,r}(t)|_{2,\Omega}^2+\nu\left( |u_{,rr}|_{2,\Omega^t}^2+|u_{,rz}|_{2,\Omega^t}^2\right) \lec D_5^2.
\eqne
\end{lemma}
\begin{proof}

We differentiate $\eqref{u_eq}_1$ with respect to $z$, multiply by $u_{,z}$ and integrate over $\Omega$ to obtain 
\eqnb\label{5.6}
\begin{split}
\frac12 \frac{\d }{\d t } |u_{,z}|_{2,\Omega}^2&-\nu\intop_\Omega \div(\nabla u_{,z}u_{,z})+\nu \intop_{\Omega}  |\nabla u_{,z}|^2 +2\nu\intop_\Omega u_{,zr}u_{,z}\d r\d z\\
&+\intop_\Omega (v_{,z}\cdot\nabla) u\cdot u_{,z}+ \frac12 \intop_\Omega v\cdot\nabla ( u_{,z}^2 )=\intop_\Omega f_{0,z}u_{,z}
\end{split}
\eqne 
We note that the second term vanishes due to the boundary condition $\left. u_{,z}\right|_S=0$ (recall \eqref{bcs}). The fourth term also vanishes since it equals
\[
\nu\intop_\Omega\partial_r(u_{,z}^2)\d r\d z=\nu\intop_{-a}^au_{,z}^2\bigg|_{r=0}^{r=R}\d z=0,
\]
where we used $u_{,z}|_{r=R}=0$ and the fact that $u_{,z}|_{r=0}=0$ (recall \eqref{2.20}). The same is true of the sixth term, which takes the form
\[
\frac12 \intop_\Omega v\cdot\nabla u_{,z}^2\d x=\frac12 \intop_Sv\cdot n \,u_{,z}^2\d S =0,
\]
since $v\cdot\bar n|_S=0$ (recall~\eqref{bcs}).  Moreover, integrating by parts in the fifth term in \eqref{5.6}, and noting that the boundary term vanishes (since $u_{,z}=0$ on $S$), we obtain  
\[
\left| \intop_\Omega (v_{,z}\cdot\nabla) u\cdot u_{,z} \right| = \left| \intop_\Omega v_{,z}\cdot\nabla u_{,z}u \right| \leq \frac{\nu}4 \intop_\Omega|\nabla u_{,z}|^2+\frac{C}{\nu}  |u|_{\infty,\Omega}^2\intop_\Omega v_{,z}^2.
\]
Finally, integrating the right-hand side of \eqref{5.6} by parts in $z$ we obtain 
\[
\left| \intop_\Omega f_{0,z}u_{,z} \right| =\left| \intop_\Omega f_0u_{,zz} \right| \leq \frac{\nu}{4} |u_{,zz}|^2_{2,\Omega }+ \frac{C}{\nu }|f_0 |_{2,\Omega }^2,
\]
since $\int_{S_2} \left[ f_0 u_{,z} \right]_{z=-a}^{z=a} r\,\d r=0$ by the boundary condition $u_{,z}|_{S_2}=0$ (recall \eqref{bcs}). 
Using the above results in \eqref{5.6} gives
\[
\frac{\d }{\d t} |u_{,z}|_{2,\Omega}^2+\nu|\nabla u_{,z}|_{2,\Omega}^2\lec \frac{1}{\nu } |u|_{\infty,\Omega}^2|v_{,z}|_{2,\Omega}^2+ \frac{1}{\nu }|f_0|_{2,\Omega}^2.
\]
Integrating in $t\in (0,T)$ gives 
\begin{equation}
|u_{,z}(t)|_{2,\Omega}^2+\nu|\nabla u_{,z}|_{2,\Omega^t}^2 \lec\frac{1}{\nu } |u|_{\infty,\Omega^t}^2|v_{,z}|_{2,\Omega^t}^2+\frac{1}{\nu } |u_{,z}(0)|_{2,\Omega}^2+\frac{1}{\nu } |f_0|_{2,\Omega^t}^2\lec D_4^2 ,
\label{5.8}
\end{equation}
proving \eqref{5.2}, where we used the energy inequality \eqref{2.1} and the maximum principle \eqref{2.9} for the swirl $u$.\\

As for \eqref{5.3} we differentiate $(\ref{u_eq})_1$ with respect to $r$, multiply the resulting equation by $u_{,r}$ and integrate over $\Omega$, to obtain
\eqnb\label{5.9}
\begin{split}
\frac12 \frac{\d }{\d t} |u_{,r}|_{2,\Omega}^2+\intop_\Omega v_{,r}\cdot\nabla uu_{,r}+\intop_\Omega v\cdot\nabla u_{,r}u_{,r}&-\nu\intop_\Omega(\Delta u)_{,r}u_{,r}\\ &+2\nu\intop_\Omega{1\over r}u_{,rr}u_{,r}-2\nu\intop_\Omega{u_{,r}^2\over r^2}=\intop_\Omega f_{0,r}u_{,r}.
\end{split}
\eqne
We now examine the particular terms in \eqref{5.9}. The second term equals
\eqnb\label{5.9a}
\begin{split}
\intop_\Omega v_{,r}\cdot\nabla u\,u_{,r}r\d r\d z&=\intop_\Omega(rv_{r,r}u_{,r}+rv_{z,r}u_{,z})u_{,r}\d r\d z\\
&=-\intop_\Omega[(rv_{r,r}u_{,r})_{,r}+(rv_{z,r}u_{,r})_{,z}]u\,\d r\d z=: I,
\end{split}
\eqne
where we integrated by parts (in $r$ and $z$ respectively) and used the boundary conditions $u|_{S_1}=u|_{r=0}=0$ (recall \eqref{bcs} and \eqref{2.20}) and $v_{z,r}|_{S_2}=0$ (recall \eqref{bcs}). Continuing, we have  
\[
I=-\intop_\Omega[(rv_{r,r})_{,r}+(rv_{z,r})_{,z}]u_{,r}u\,\d r\d z -\intop_\Omega[rv_{r,r}u_{,rr}+rv_{z,r}u_{,rz}]u\,\d r\d z=:  I_1+I_2.
\]
We now note that differentiating the divergence-free condition $\eqref{nse_cylin}_4$ in $r$ gives $
v_{r,rr}+v_{z,zr}+{v_{r,r}\over r}-{v_r\over r^2}=0
$, which shows that the part of the integrand of $I_1$ in the square brackets equals to $v_r/r$, and so
\[I_1=-\intop_\Omega{v_r\over r}u_{,r}u\,\d r\d z.\]
As for $I_2$ we use Young's inequality to obtain
\[
|I_2|\le \frac{\nu}2 (|u_{,rr}|_{2,\Omega}^2+|u_{,rz}|_{2,\Omega}^2)+\frac{C}{\nu } |u|_{\infty,\Omega}^2 (|v_{r,r}|_{2,\Omega}^2+|v_{z,r}|_{2,\Omega}^2).
\]
The third term on the left-hand side of \eqref{5.9}  equals
$$
{1\over 2}\intop_\Omega v\cdot\nabla u_{,r}^2={1\over 2}\intop_\Omega\divv(vu_{,r}^2)=0
$$
since  $v\cdot\bar n|_S=0$. As for the fourth term in \eqref{5.9}, we have
\eqnb
\begin{split}
-\intop_\Omega&(\Delta u)_{,r}u_{,r}=-\intop_\Omega\bigg(u_{,rrr}+\bigg({1\over r}u_{,r}\bigg)_{,r}+u_{,rzz}\bigg)u_{,r}r\d r\d z\\
&=-\intop_\Omega\bigg[\bigg(u_{,rr}+{1\over r}u_{,r}\bigg)u_{,r}r\bigg]_{,r}\d r\d z+\intop_\Omega u_{,rr}(u_{,r}r)_{,r}\d r\d z+\intop_\Omega{1\over r}u_{,r}(u_{,r}r)_{,r}\d r\d z+\intop_\Omega u_{,rz}^2\\
&=-\intop_{-a}^a\left[ \bigg(u_{,rr}+{1\over r}u_{,r}\bigg)u_{,r}r\right]_{r=0}^{r=R}\d z+\intop_\Omega (u_{,rr}^2+u_{,rz}^2)+\intop_\Omega{u_{,r}^2\over r^2}+2\intop_\Omega u_{,rr}u_{,r}\d r\d z.
\end{split}
\eqne
Using the above expressions in (\ref{5.9}) yields
\eqnb\label{5.13}
\begin{split}
\frac12 \frac{\d }{\d t} |u_{,r}|_{2,\Omega}^2&+\frac{\nu}2\intop_\Omega\left( u_{,rr}^2+u_{,zr}^2\right)-\nu\intop_\Omega{u_{,r}^2\over r^2}-\nu\intop_{-a}^a\left[ \left(u_{,rr}+{1\over r}u_{,r}\right)u_{,r}r\right]_{r=0}^{r=R}\d z+4\nu\intop_\Omega u_{,rr}u_{,r}\d r\d z\\
&\le \intop_\Omega f_{0,r}u_{,r}+\intop_\Omega{v_r\over r}{u_{,r}\over r}u+ \frac{C}{\nu } |u|_{\infty,\Omega}^2 \left( |v_{r,r}|_{2,\Omega}^2+|v_{z,r}|_{2,\Omega}^2\right).
\end{split}
\eqne
The last term on the left-hand side of \eqref{5.13} equals
$$
2\nu\intop_{-a}^au_{,r}^2\bigg|_{r=0}^{r=R}\d z=2\nu\intop_{-a}^au_{,r}^2\bigg|_{r=R}\d z
$$
since the expansion \eqref{2.20} implies that 
\eqnb\label{exp_u}
u=b_1(z,t)r^2+b_2(z,t)r^4+\cdots,
\eqne
so that $u_{,r}|_{r=0}=0$. Moreover, the above expansion of $u$ lets us examine the fourth term on the left-hand side of \eqref{5.13} to obtain
\eqnb\label{001}
-\nu\intop_{-a}^a\bigg(u_{,rr}+{1\over r}u_{,r}\bigg)u_{,r}r\bigg|_{r=R}\d z = -2\nu\intop_{-a}^a \left[ u_{,r}^2\right]^{r=R}_{r=0}\d z+\intop_{-a}^a \left. f_0u_{,r}r\right|_{r=R}\d z,
\eqne
where, in the last equality, we used the swirl equation $\eqref{u_eq}_1$ projected onto $S_1$. 

As for the first term on the right-hand side of \eqref{5.13}, integration by parts in $r$ gives
\[
\intop_{-a}^a\left. f_0u_{,r}r\right|_{r=R}\d z -\intop_\Omega f_0u_{,rr}-\intop_\Omega f_0u_{,r}\d r\d z,
\]
where we used \eqref{exp_u} again to note that $u_{,r}|_{r=0}=0$. We note that the first term above cancels with the last term of \eqref{001}, while the remaining terms can be estimated using Young's inequality by 
$$
\frac{\nu}4 |u_{,rr}|_{2,\Omega}^2+ \nu \left|{u_{,r}\over r}\right|_{2,\Omega}^2+\frac{C}{\nu} |f_0|_{2,\Omega}^2
$$
Using the above estimates in (\ref{5.13}) and simplifying the result we get
\[
\begin{split}
&\frac12 \frac{\d }{\d t} |u_{,r}|_{2,\Omega}^2+{\nu\over 4}\intop_\Omega(u_{,rr}^2+u_{,rz}^2)\le 2\nu \intop_\Omega{u_{,r}^2\over r^2}+\intop_\Omega{v_r\over r}{u_{,r}\over r}u+c|u|_{\infty,\Omega}^2(|v_{r,r}|_{2,\Omega}^2+|v_{z,r}|_{2,\Omega}^2)+\frac{C}{\nu}|f_0|_{2,\Omega}^2.
\end{split}
\]
Integrating with respect to $t\in (0,T)$, using the maximum principle (Lemma~\ref{l2.4}) and the energy inequality \eqref{2.1}, which in particular implies that 
\eqnb\label{5.4}
\nu \intop_{\Omega^t} \frac{u_{,z}^2+u_{,r}^2}{r^2} \lec \nu \intop_{\Omega^t}  \left( v_{\varphi,z}^2 + v_{\varphi,r}^2 + \frac{v_{\varphi}^2}{r^2} \right) \lec D_1^2,
\eqne
we obtain
\[|u_{,r}(t)|_{2,\Omega}^2+\nu\left( |u_{,rr}|_{2,\Omega^t}^2+|u_{,zr}|_{2,\Omega^t}^2\right)\lec D_1^2+\frac{1}{\nu} D_2D_1^2+\frac{1}{\nu}|f_0|_{2,\Omega^t}^2+|u_{,r}(0)|_{2,\Omega}^2\leq D_5^2,\]
which shows (\ref{5.3}), as required. 
\end{proof}

\section{Order reduction estimates}\label{sec_order_red}
Here we prove the order reduction estimates of Theorem~\ref{T1}. We first consider \eqref{est1}.
\begin{lemma}[Order reduction for $\omega_r$, $\omega_z$]\label{l6.1}
Any regular solution $v$ to \eqref{nse}--\eqref{bcs}  satisfies 
\eqnb\label{6.1}
\begin{split}
\|\omega_r\|_{V(\Omega^t)}^2+\|\omega_z\|_{V(\Omega^t)}^2&+\bigg|{\omega_r\over r}\bigg|_{2,\Omega^t}^2 \\
&\leq C_{D_1,D_2,D_4,D_5} \bigg[{R^{2\delta }\over\delta^2} |v_\varphi|_{\infty,\Omega^t}^{2\delta }+{R^{\delta }\over\delta }+1\bigg] \left( |\Gamma_{,z}|_{2,\Omega^t} +|\Gamma_{,r}|_{2,\Omega^t} \right) +cD_8^2
\end{split}
\eqne
for every $\delta \in (0,1)$.
\end{lemma}

\begin{proof}
We multiply $\eqref{1.9}_1$ by $\omega_r$, $\eqref{1.9}_3$ by $\omega_z$, we add the resulting equalities and integrate over $\Omega^t$ to obtain
\eqnb\label{6.2}
\begin{split}
&{1\over 2}\left( |\omega_r(t)|_{2,\Omega}^2+|\omega_z(t)|_{2,\Omega}^2\right) +\nu \left( |\nabla\omega_r|_{2,\Omega^t}^2+|\nabla\omega_z|_{2,\Omega^t}^2+\left|\frac{\omega_r}r\right|_{2,\Omega^t}^2\right) \underbrace{-\nu\intop_{S^t} \left( n\cdot\nabla\omega_r\omega_r + n\cdot\nabla\omega_z\omega_z\right) }_{=:I}\\
&\hspace{1.5cm}=\underbrace{\intop_{\Omega^t}\left( v_{r,r}\omega_r^2+v_{z,z}\omega_z^2+(v_{r,z}+v_{z,r})\omega_r\omega_z\right)}_{=:J}+\intop_{\Omega^t}(F_r\omega_r+F_z\omega_z) +{1\over 2}(|\omega_r(0)|_{2,\Omega}^2+|\omega_z(0)|_{2,\Omega}^2).
\end{split}
\eqne
We will show that
\eqnb\label{step1}
|I| \lec \|f_\varphi\|_{L_2(0,t;L_2(S_1))}(D_4+D_5) 
\eqne
in Step 1 below, and that
\eqnb\label{step2}
\begin{split}
|J| &\leq \frac{\nu}2 \left( \left|\omega_{r,z} \right|_{2,\Omega^t}^2+\left| \omega_{z,r} \right|_{2,\Omega^t}^2 \right) \\
&\hspace{2cm} +  \frac{C_{D_1,D_2,D_4,D_5}}{\nu} \left( \frac{R^{\delta }}{\delta } |v_\varphi |^{\delta }_{\infty,\Omega^t}  +  \frac{R^{2\delta }}{\delta^2} |v_\varphi |^{2\delta }_{\infty,\Omega^t} +1 \right)\left(   \left| \Gamma_{,r} \right|_{2,\Omega^t} + \left| \Gamma_{,z} \right|_{2,\Omega^t}\right)
\end{split}
\eqne
for every $\delta \in (0,1/2)$, $\nu>0 $ in Step 2 below. The claim then follows by applying these in \eqref{6.2} and estimating the remaining terms on the right-hand side by the Cauchy-Schwarz inequality and the energy inequality \eqref{2.1}.\\

\noindent\texttt{Step 1.} We show \eqref{step1}.\\

We note that all the boundary terms included in $I$ vanish, except for  
\eqnb\label{6.2a}
 -\nu\intop_{S_1^t} n\cdot\nabla\omega_z\omega_z,
\eqne
since $\omega_r|_{S}=0$ and $n\cdot \nabla \omega_z = \omega_{z,z}= 0 $ on $S_2$ (recall \eqref{bcs1}). As for \eqref{6.2a}, we note that $\omega_z=v_{\varphi,r}+v_\varphi/ r$ (recall $\eqref{1.13}_3$), and so
\[\begin{split}
I  &= -\nu R\intop_0^t\intop_{-a}^a\partial_r \bigg(v_{\varphi,r}+{v_\varphi\over r}\bigg)\bigg(v_{\varphi,r}+{v_\varphi\over r}\bigg)\bigg|_{S_1}\d z \d \ta = -\nu R\intop_0^t\intop_{-a}^a \bigg(v_{\varphi,rr}+{v_{\varphi,r}\over r}\bigg) v_{\varphi,r}\bigg|_{S_1}\d z \d \ta \\
&= R\intop_0^t\intop_{-a}^a \left. f_\varphi v_{\varphi,r}\right|_{r=R}\d z \d \ta =\intop_0^t\intop_{-a}^af_\varphi\bigg(u_{,r}-{1\over R}u\bigg)\bigg|_{r=R}\d z \d \ta ,
\end{split}
\]
where we noted that $v_\varphi|_{r=R}=0$ (recall \eqref{bcs}) in the second equality, and that 
$$
-\nu\bigg(v_{\varphi,rr}+{1\over r}v_{\varphi,r}\bigg)=f_\varphi\quad {\rm on}\ \ S_1
$$
(a consequence of projecting the equation $\eqref{nse_cylin}_2$ for $v_\varphi$ onto $S_1$) in the third equality.
 
Hence, applying the H\"older inequality yields
\[
|I |\lec \intop_0^t|f_\varphi|_{2,S_1}(|u_{,r}|_{2,S_1}+|u|_{2,S_1})\d \ta \lec \|f_\varphi\|_{L_2(0,t;L_2(S_1))}(D_4+D_5) ,
\]
as required, where we used trace estimates and Lemma \ref{l5.1} in the last step. \\

\noindent\texttt{Step 2.} We show \eqref{step2}.\\

To this end we recall \eqref{1.13}, \eqref{1.15} that $v_{r,r}=-\psi_{,zr}$, $\omega_r=-u_{,z}/r$, $v_{z,z}= \psi_{,rz}+ \psi_{,z}/r$, $\omega_z= u_{,r}/r$ $v_{r,z}=-\psi_{zz}$, $v_{z,r}=\psi_{,rr}+\psi_{,r}/r-\psi / r^2$, so that 
\[J=\intop_{\Omega^t}\left[-\psi_{,zr}\frac{u_{,z}^2 }{r^2} +\left(\psi_{,rz}+{\psi_{,z}\over r}\right)\frac{u_{,r}^2}{r^2} -\left(-\psi_{,zz}+\psi_{,rr}+{1\over r}\psi_{,r}-{\psi\over r^2}\right)\frac{u_{,z}u_{,r}}{r^2}  \right]=:  J_1+J_2+J_3.\]
For $J_1$, we integrate by parts with respect to $z$ and use that $u_{,z}|_{S_2}=0$ (recall \eqref{bcs1}) to obtain
\[
J_1=\intop_{\Omega^t}\psi_{,zzr}{1\over r^2}u_{,z}u +\intop_{\Omega^t}\psi_{,zr}{1\over r^2}u_{,zz}u =: J_{11} + J_{12} .\]
Since $\psi=r\psi_1$,
\[
J_{11}=\intop_{\Omega^t}\bigg({\psi_{1,zz}\over r}+\psi_{1,zzr}\bigg){u_{,z}\over r}u =: J_{111} + J_{112}.
\]
We estimate $J_{111}$ by writing 
\[
\begin{split}
|J_{111}|&\le\intop_{\Omega^t}\bigg|{\psi_{1,zz}\over r^{1-\delta }}\bigg|\,\bigg|{u_{,z}\over r}\bigg|\,|v_\varphi|^{\delta }|u|^{1-\delta }\\
&\le D_2^{1-\delta }|v_\varphi|_{\infty,\Omega^t}^{\delta }\bigg|{u_{,z}\over r}\bigg|_{2,\Omega^t}\bigg|{\psi_{1,zz}\over r^{1-\delta }}\bigg|_{2,\Omega^t}\\
&\lec \frac{ D_1 D_2^{1-\delta }R^{\delta }}{\nu^{1/2}\delta }|v_\varphi|_{\infty,\Omega^t}^{\delta }|\Gamma_{,z}|_{2,\Omega^t} 
\end{split}
\]
for every $\delta \in (0,1)$, where we used the Cauchy-Schwarz inequality in the second inequality, and \eqref{5.4} as well as \eqref{2.26} in the form 
$$
\bigg|{\psi_{1,zz}\over r^{1-\delta }}\bigg|_{2,\Omega^t}\lec {R^{\delta }\over\delta }|\Gamma_{,z}|_{2,\Omega^t}
$$
in the third inequality. As for $J_{112}$, the Cauchy-Schwarz inequality gives 
\begin{equation}
|J_{112}|\le|u|_{\infty,\Omega^t}\bigg|{u_{,z}\over r}\bigg|_{2,\Omega^t}|\psi_{1,zzr}|_{2,\Omega^t}\lec \frac{D_1D_2}{\nu^{1/2}} |\Gamma_{,z}|_{2,\Omega^t},
\label{6.8}
\end{equation}
where we used \eqref{5.4}, the maximum principle for $u$ (Lemma~\ref{l2.4}) and \eqref{3.13}. 

For $J_{12}$ we also recall that $\psi=r\psi_1$ to write
$$
J_{12}=\intop_{\Omega^t}\bigg({\psi_{1,z}\over r^2}+{\psi_{1,zr}\over r}\bigg)u_{,zz}u=: J_{121}+J_{122},
$$
and we estimate $J_{121}$ by
\eqnb\label{j121}
\begin{split}
|J_{121}| &\le \frac{\nu}2 \left|\frac{u_{,zz}}r \right|_{2,\Omega^t}^2+\frac{c}{\nu } \left| \frac{\psi_{1,z}u }r \right|^2_{2,\Omega^t }\\
&\le \frac{\nu}2 \left|\omega_{r,z} \right|_{2,\Omega^t}^2+\frac{c}{\nu } \left| u\right|^{2(1-\delta  )}_{\infty,\Omega^t } | v_\varphi |_{\infty,\Omega^t}^{2\delta } \left| \frac{\psi_{1,z} }{r^{1-\delta }} \right|^2_{2,\Omega^t }\\
&\le \frac{\nu}2 \left|\omega_{r,z} \right|_{2,\Omega^t}^2+\frac{cR^{2\delta  }}{\nu \delta^2 } D_1 D_2^{2(1-\delta  )}| v_\varphi |_{\infty,\Omega^t}^{2\delta } |\Gamma_{,z} |_{2,\Omega^t} ,
\end{split}
\eqne
where we used, in the second line, the fact that $\omega_r = u_{,z}/r$ (recall~\eqref{1.13}), and, in the third line, the maximum principle \eqref{2.9} and the inequality 
\eqnb\label{psi1z_r_est}
\left| \frac{\psi_{1,z} }{r^{1-\delta }} \right|^2_{2,\Omega^t } \lec \frac{R^{2\delta }}{\delta^2 } |\psi_{1,rz}|^2_{2,\Omega^t}  \lec \frac{R^{2\delta }}{\delta^2 } \left| D^2_{r,z} \psi_{1,z} \right|_{2,\Omega^t }\left| \psi_{1,z} \right|_{2,\Omega^t } \lec  \frac{D_1 R^{2\delta }}{\delta^2 } \left| \Gamma_{,z} \right|_{2,\Omega^t } ,
\eqne
which is a consequence of the Hardy inequality (i.e. Lemma~\ref{l2.6}, applied with $\beta\coloneqq 1/2-\delta $),  the interpolation inequality \eqref{2.17}, and  the estimates $|\psi_{1,z}|_{2,\Omega^t} \lec D_1$ (recall~\eqref{2.15}) and $|D^2_{rz} \psi_{1,z}|_{2,\Omega^t} \lec |\Gamma_{,z} |_{2,\Omega^t}$ (recall~\eqref{3.14}).

On the other hand, for $J_{122}$, we obtain  
\begin{equation}
|J_{122}|\le|u|_{\infty,\Omega^t}|u_{,zz}|_{2,\Omega^t}\bigg|{\psi_{1,zr}\over r}\bigg|_{2,\Omega^t}\lec \frac{ D_2D_4}{\nu^{1/2}} |\Gamma_{,z}|_{2,\Omega^t},
\label{6.9}
\end{equation}
where we used the maximum principle \eqref{2.9},  \eqref{5.2} and the $H^3$ estimate \eqref{3.24} of $\psi$. 

For $J_2$, we have
\[
\begin{split}
J_2&=\intop_{\Omega^t}\bigg(\psi_{,rz}+{\psi_{,z}\over r}\bigg){1\over r}u_{,r}u_{,r}\d r\d z\d \ta\\
&=-\intop_{\Omega^t}\bigg(\psi_{,rz}+{\psi_{,z}\over r}\bigg)_{,r}{1\over r}u_{,r}u\,\d r\d z\d \ta-\intop_{\Omega^t}\bigg(\psi_{,rz}+{\psi_{,z}\over r}\bigg)\bigg({1\over r}u_{,r}\bigg)_{,r}u\,\d r\d z\d \ta \\
&=-\intop_{\Omega^t}\left(\psi_{,rz}+{\psi_{,z}\over r}\right)_{,r}\frac{u_{,r}}{r^2} u- \intop_{\Omega^t}\left(\psi_{,rz}+{\psi_{,z}\over r}\right){1\over r}\left(\frac{u_{,r}}r \right)_{,r}u =: J_{21}+J_{22},
\end{split}\]
where, in the second line, we integrated by parts with respect to $r$ and observed that both boundary terms vanish, since  $u|_{r=R}=0$ and 
$$
\bigg(\psi_{,rz}+{\psi_{,z}\over r}\bigg){1\over r}u_{,r}u\bigg|_{r=0}=0,
$$
due to $u|_{r=0}=0$ and the expansions near $r=0$: 
\[
\begin{cases}
&\psi=a_1(z,t)r+a_2(z,t)r^3+\dots,\\
&v_\varphi=b_1(z,t)r+b_2(z,t)r^3+\cdots 
\end{cases} \Rightarrow \begin{cases} 
&\psi_{,rz}\sim a_{1,z}, \\
&{\psi_{,z}\over r}\sim a_{1,z}, \\
&{1\over r}u_{,r}={v_\varphi\over r}+v_{\varphi,r}\sim 2b_1,
\end{cases}
\]
see \cite{LW}. For $J_{21}$ we use that $\psi=r\psi_1$ to obtain 
\[
\begin{split}
\left| J_{21} \right| &= \left| \intop_{\Omega^t}(3\psi_{1,rz}+r\psi_{1,rrz})\frac{u_{,r}}{r^2} u \right| \\
&\leq 3\left|{\psi_{1,rz}\over r}\right|_{2,\Omega^t}\left|\frac{u_{,r}}{r} \right|_{2,\Omega^t} |u|_{\infty,\Omega^t}+|\psi_{1,rrz}|_{2,\Omega^t}\left| \frac{u_{,r}}{r} \right|_{2,\Omega^t}|u|_{\infty,\Omega^t}\\
&\lec \frac{D_1D_2}{\nu^{1/2}} |\Gamma_{,z}|_{2,\Omega^t},
\end{split}
\]
where we used \eqref{3.24}, \eqref{3.14}, \eqref{5.4} and \eqref{2.9}.

For $J_{22}$, we again use that $\psi=r \psi_1$ to get
\[
\begin{split}
|J_{22}|&=\left| \int_{\Omega^t}(2\psi_{1,z}+r\psi_{1,rz}){1\over r}\left( \frac{u_{,r}}{r} \right)_{,r}u \right|\\
&\leq 2 \left| \int_{\Omega^t}{\psi_{1,z}\over r}\bigg({1\over r}u_{,r}\bigg)_{,r}u \right|+ \left|\int_{\Omega^t}\psi_{1,rz}\bigg({1\over r}u_{,r}\bigg)_{,r}u\right|\\
& \leq \frac{\nu}2 \left| \left( \frac{u_{,r}}r \right)_{,r} \right|_{2,\Omega^t}^2 + \frac{c}{\nu}   \left| \frac{\psi_{1,z} u }r \right|_{2,\Omega^t}^2+ \left|{\psi_{1,rz}\over r}\right|_{2,\Omega^t} \left( |u_{,rr}|_{2,\Omega^t}|u|_{\infty,\Omega^t} + \left|{1\over r}u_{,r}\right|_{2,\Omega^t}|u|_{\infty,\Omega^t}\right) \\
& \leq \frac{\nu}2 \left| \omega_{z,r} \right|_{2,\Omega^t}^2 +c \left( \frac{R^{2\delta  }}{\nu \delta^2 } D_1 D_2^{2(1-\delta  )}| v_\varphi |_{\infty,\Omega^t}^{2\delta }  + \frac{D_2}{\nu^{1/2}} (D_1+D_5) \right)|\Gamma_{,z}|_{2,\Omega^t} ,
\end{split}
\]
where, in the last line, we recalled~\eqref{1.13} that $(u_{,r}/r)_{,r}=\omega_{z,r}$, used the same estimate as in \eqref{j121} to control the second term, and we used \eqref{3.24}, the energy inequality \eqref{5.4}, the maximum principle \eqref{2.9} as well as to control the last term. 

For $J_3$, we use again that $\psi=r\psi_1$ to get
\[
\begin{split}
J_3&=-\intop_{\Omega^t}(-r\psi_{1,zz}+3\psi_{1,r}+r\psi_{1,rr}){1\over r}u_{,r}{1\over r}u_{,z} \\
&= \underbrace{\intop_{\Omega^t}\bigg(-\psi_{1,zzz}+{3\over r}\psi_{1,rz}+\psi_{1,rrz}\bigg){1\over r}u_{,r}u}_{=: J_{31}}+\intop_{\Omega^t}\bigg(-\psi_{1,zz}+{3\over r}\psi_{1,r}+\psi_{1,rr}\bigg){1\over r}u_{,rz}u=: J_{31} + J_{32} ,
\end{split}
\] where, in the second line we integrated by parts with respect to $z$ and used that $\psi_1|_{S_2}=\psi_{1,zz}|_{S_2}=0$ (by \eqref{psi1_eq} and \eqref{bc_psi1,zz}) to deduce that the boundary terms vanish, and, in the last line, we recalled \eqref{psi1_eq} to observe that $-\psi_{1,zz}+3 \psi_{1,r}/r+\psi_{1,rr}=-\Gamma-2\psi_{1,zz}$.
Clearly 
\begin{equation}
|J_{31}|\lec \frac{D_1D_2}{\nu^{1/2}}|\Gamma_{,z}|_{2,\Omega^t},
\label{6.15}
\end{equation}
due to \eqref{2.9}, \eqref{3.13}, \eqref{3.14}, \eqref{3.24} and the energy inequality \eqref{5.4}. 

For $J_{32}$ we have
\[\begin{split}
|J_{32}| &= \left|  \intop_{\Omega^t} \frac{\Gamma u_{,rz}u}r + 2\intop_{\Omega^t}\frac{\psi_{1,zz} u_{,rz}u}r \right| \\
&\leq   | u |_{\infty, \Omega^t}^{1-\delta } | v_\varphi |_{\infty, \Omega^t }^{\delta } \left| u_{,rz} \right|_{2,\Omega^t} \left( \left| \frac{\Gamma}{r^{1-\delta }} \right|_{2,\Omega^t}+2\left| \frac{\psi_{1,zz}}{r^{1-\delta }} \right|_{2,\Omega^t} \right)  \\
&\lec  \frac{D_2^{1-\delta } D_5}{ \nu^{1/2}} | v_\varphi |_{\infty, \Omega^t }^{\delta }\left(  \frac{R^{\delta }}{\delta  } \left| \Gamma_{,r} \right|_{2,\Omega^t} +\frac{R^{\delta }}{\delta^{1/2} } \left| \Gamma_{,z} \right|_{2,\Omega^t}\right) ,
\end{split}
\]
as needed, where we recalled \eqref{psi1_eq} to observe that $-\psi_{1,zz}+3 \psi_{1,r}/r+\psi_{1,rr}=-\Gamma-2\psi_{1,zz}$ in the first equality, and, in the second inequality, we used the maximum principle \eqref{2.9} and the swirl estimate \eqref{5.3}  in the first inequality,  the Hardy inequality \eqref{2.13} in the form
\[
\int_0^R \frac{\Gamma^2}{r^{2(1-\delta )}} r \d r \lec \delta^{-2} \int_0^R r^{2\delta } \Gamma_{,r}^2 r \d r \leq \frac{R^{2\delta }}{\delta^{2}} \int_0^R  \Gamma_{,r}^2 r \d r ,
\]
and the weighted $\psi_{1}$ estimate \eqref{2.26} in the form $|\psi_{1,zz} r^{\delta  -1} |_{2,\Omega^t } \lec \delta ^{-1/2} |\Gamma_{,z}r^{\delta }|_{2,\Omega^t} $
for $\delta \in (0,1/2)$.
\end{proof}

We now focus on \eqref{est2} and \eqref{est3}. 
\begin{lemma}[Order reduction estimates for $v_\varphi$]\label{l4.5}
For every regular solution $v$ to \eqref{nse}--\eqref{bcs}  
\eqnb\label{4.46}
|v_\varphi |_{\infty , \Omega^t} \lec \frac{D_1^{1/4}}{\nu^{1/2}} X^{3/4} + D_{10}.
\eqne
Moreover, if $s\in (1,\infty )$ and \eqref{4.29} holds, i.e.
\[
\frac{|v_\varphi|_{s,\infty,\Omega^t} }{|v_\varphi|_{\infty,\Omega^t}} \ge c_0,
\]
for some constant $c_0\in (0,\infty)$, then 
\eqnb\label{4.31}
|v_\varphi |_{s, \infty, \Omega^t} \lec D_{11} X^{1/(4-2\delta )} + D_{12}  .
\eqne
\end{lemma}
\begin{proof}
We multiply  the equation $(\ref{nse_cylin})_2$ for $v_\varphi$ by $v_\varphi|v_\varphi|^{s-2}$, integrate over $\Omega$ and use the fact that ${v_r\over r}=-\psi_{1,z}$ to  obtain
\eqnb\label{4.32}
\begin{split}
{1\over s}\frac{\d }{\d t} |v_\varphi|_{s,\Omega}^s&+{4\nu(s-1)\over s^2}|\nabla|v_\varphi|^{s/2}|_{2,\Omega}^2+\nu\intop_\Omega{|v_\varphi|^s\over r^2}=\intop_\Omega\psi_{1,z}|v_\varphi|^s+\intop_\Omega f_\varphi v_\varphi|v_\varphi|^{s-2}
\end{split}
\eqne
We estimate the first term on the right-hand side by $\nu \int_\Omega |v_\varphi |^s/r^2 + (4\nu)^{-1} \int_\Omega r^2 |\psi_{1,z} |^2 | v_\varphi |^s$, and thus
\[
{1\over s}\frac{\d }{\d t} |v_\varphi|_{s,\Omega}^s\lec \frac{1}{\nu} \intop_\Omega r^2 |\psi_{1,z} |^2 | v_\varphi |^s + \intop_\Omega r\left| \frac{f_\varphi}r \right| |v_\varphi|^{s-1} \leq \frac{D_2^2}{\nu} |\psi_{1,z} |_{s,\Omega}^2 | v_\varphi |_{s,\Omega}^{s-2} + D_2 \left| \frac{f_\varphi}r \right|_{\frac{s}2,\Omega} |v_\varphi|^{s-2}_{s,\Omega} 
\]
where we used the maximum principle \eqref{2.9} in the second inequality. Thus, we can divide by $|v_\varphi|_{s,\Omega}^{s-2}$ and integrate in time to obtain
\[
|v_{\varphi } (t) |_{s,\Omega}^2 - |v_{\varphi } (0) |_{s,\Omega}^2 \lec \frac{1}{\nu } \int_0^t |\psi_{1,z} |_{s,\Omega}^2 + D_{10}^2
\]
for every $t\in (0,T)$. Taking $s\to \infty$ and using the interpolation inequality \eqref{2.17} in the form $|g|_{\infty,\Omega} \lec | g |_{2,\Omega}^{1/4}|D^2 g |_{2,\Omega}^{3/4} + | g |_{2,\Omega}$ we obtain 
\[
|v_{\varphi } (t) |_{\infty,\Omega}^2 \lec |v_{\varphi } (0) |_{\infty,\Omega}^2+ \frac{1}{\nu } \int_0^t \left( |\psi_{1,z} |_{2,\Omega}^{1/2}|D^2 \psi_{1,z} |_{2,\Omega}^{3/2} + |\psi_{1,z} |_{2,\Omega}^2 \right) + D_{10}^2\lec \frac{{D_1}^{1/2}}{\nu } |\Gamma_{,z}|_{2,\Omega^t}^{3/2} + D_{10}^2,
\]
from which \eqref{4.46} follows.

As for \eqref{4.31}, we multiply  the equation $(\ref{nse_cylin})_2$ for $v_\varphi$ by $v_\varphi|v_\varphi|^{s-2}$, integrate over $\Omega$ and use the fact that ${v_r\over r}=-\psi_{1,z}$ to  obtain
\eqnb\label{4.32}
\begin{split}
{1\over s}\frac{\d }{\d t} |v_\varphi|_{s,\Omega}^s&+{4\nu(s-1)\over s^2}|\nabla|v_\varphi|^{s/2}|_{2,\Omega}^2+\nu\intop_\Omega{|v_\varphi|^s\over r^2}=\intop_\Omega\psi_{1,z}|v_\varphi|^s+\intop_\Omega f_\varphi v_\varphi|v_\varphi|^{s-2}
\end{split}
\eqne
We estimate the first term on the right-hand side by
\[\begin{split}
& \nu \int_\Omega |v_\varphi |^s/r^2 + (4\nu)^{-1} \int_\Omega r^2 |\psi_{1,z} |^2 | v_\varphi |^s\\
&\leq \nu \int_\Omega |v_\varphi |^s/r^2 + \frac{D_2^{4-2\delta }}{4\nu} \int_\Omega \frac{\psi_{1,z}^2}{r^{2-2\delta }}  | v_\varphi |^{s-4+2\delta }, 
\end{split} \]
where $\delta >0$ is small, and we used the maximum principle \eqref{2.9}. Since also
\[
\int_\Omega |f_\varphi |\, |v_\varphi |^{s-1} \leq | v_\varphi |_{\infty, \Omega }^{s-4+2\delta } \int_\Omega | f_\varphi |\, |v_\varphi |^{3-2\delta }
\]
we thus have
\[
{1\over s}\frac{\d }{\d t} |v_\varphi|_{s,\Omega}^s\lec \frac{D_2^{4-2\delta}}{4\nu } | v_\varphi |_{\infty , \Omega }^{s-4+2\delta } \int_\Omega \frac{\psi_{1,z}^2}{r^{2-2\delta }} +| v_\varphi |_{\infty, \Omega }^{s-4+2\delta } \int_\Omega | f_\varphi |\, |v_\varphi |^{3-2\delta }
\]
Noting that
\[
\frac1s \frac{\d }{\d t} |v_\varphi |_{s,\Omega }^s = |v_\varphi |_{s,\Omega }^{s-1} \frac{\d }{\d t} |v_\varphi |_{s,\Omega } = |v_\varphi |_{s,\Omega }^{s-4+2\delta } |v_\varphi |_{s,\Omega }^{3-2\delta } \frac{\d }{\d t } |v_\varphi |_{s,\Omega } = | v_\varphi |^{s-4+2\delta } \frac{1}{4-2\delta } \frac{\d }{\d t} |v_\varphi |_{s,\Omega }^{4-2\delta }
\]
Thus, making use of the assumption \eqref{4.29} we obtain that
\[
\frac{1}{4-2\delta } \frac{\d }{\d t} |v_\varphi |_{s,\Omega }^{4-2\delta } \leq c_0^{-s+4-2\delta } \left( \frac{D_2^{4-2\delta }}{4\nu } \int_\Omega \frac{\psi_{1,z}^2}{r^{2-2\delta }} + \int_\Omega |f_\varphi |\, |v_\varphi |^{3-2\delta } \right) 
\]
Integrating in time we obtain
\eqnb\label{estaax}
\begin{split}
|v_\varphi (t) |_{s,\Omega }^{4-2\delta } &\leq \frac{4-2\delta }{c_0^{s-4+2\delta }} \left( \frac{D_2^{4-2\delta }}{4\nu} \int_{\Omega^t} \frac{\psi_{1,z}^2}{r^{2-2\delta} } + \int_{\Omega^t } |f_\varphi|\, | v_\varphi |^{3-2\delta } \right) \\
&\leq \frac{4-2\delta }{c_0^{s-4+2\delta }} \left( \frac{D_2^{4-2\delta }}{4\nu} \int_{\Omega^t} \frac{\psi_{1,z}^2}{r^{2-2\delta }} + | f_\varphi |_{10/(1+6\delta ), \Omega^t  } | v_\varphi |_{10/3,\Omega^t }^{3-2\delta } \right)   ,
\end{split}
\eqne
where we also used H\"older's inequality in the second line.  Note that $| v_\varphi |_{10/3,\Omega^t } \leq D_1$ and
\[
\int_{\Omega^t} \frac{\psi_{1,z}^2}{r^{2-2\delta }} \lec \frac{ R^{2\delta } }{\delta^2 }  |\psi_{1,zr} |^2_{2,\Omega^t} \leq  \frac{ R^{2\delta } }{\delta^2 }  D_1 X(t),
\]
where we used the Hardy inequality \eqref{2.13} in the first inequality and \eqref{psi1_h2} in the second inequality. Thus, taking the $1/(4-2\delta )$ power of \eqref{estaax} we obtain
\[
|v_\varphi (t)|_{s,\Omega } \leq D_{11} X^{\frac{1}{4-2\delta }} + D_{12},
\]
as required. Recall Section~\ref{s1} for the definitions of $D_{11}$, $D_{12}$.
\end{proof}

\section{Global estimate for regular solutions}\label{sec_reg}

Assuming appropriate regularity of data we show that boundedness of $X(t)$ implies that 
\eqnb\label{7.1}
\|v\|_{W_2^{4,2}(\Omega^t)}+\|\nabla p\|_{W_2^{2,1}(\Omega^t)} \leq C,
\eqne 
where $C>0$ depends only on initial data and the forcing.

\subsection{Preliminaries}
We first introduce some functional analytic tools to handle the anisotropic Sobolev spaces. 

\begin{definition}[Anisotropic Sobolev and Sobolev-Slobodetskii spaces]\label{d2.11}
We denote by
\begin{itemize}
\item[1.] $W_{p,p_0}^{k,k/2}(\Omega^T)$, $k,k/2\in\N\cup\{0\}$, $p,p_0\in[1,\infty]$ -- the anisotropic Sobolev space with a mixed norm, which is a completion of $C^\infty(\Omega^T)$-functions under the norm
$$
\|u\|_{W_{p,p_0}^{k,k/2}(\Omega^T)}=\left(\int_0^T\left(\sum_{|\alpha|+2\alpha\le k}\int_\Omega |D_x^\alpha\partial_t^au|^p\right)^{p_0/p}\d t\right)^{1/p_0}.
$$
\item[2.] $W_{p,p_0}^{s,s/2}(\Omega^T)$, $s\in\R_+$, $p,p_0\in[1,\infty)$ -- the Sobolev-Slobodetskii space with the finite norm
$$\eqal{
&\|u\|_{W_{p,p_0}^{s,s/2}(\Omega^T)}=\sum_{|\alpha|+2a\le[s]}\|D_x^\alpha\partial_t^au\|_{L_{p,p_0}(\Omega^T)}\cr
&\quad+\bigg[\intop_0^T\!\!\bigg(\intop_\Omega\!\intop_\Omega\!\sum_{|\alpha|+2a= [s]}\hskip-10pt{|D_x^\alpha\partial_t^au(x,t)-D_{x'}^\alpha\partial_t^au(x',t)|^p\over |x-x'|^{n+p(s-[s])}}\d x\d x'\bigg)^{p_0/p}\hskip-5pt \d t\bigg]^{1/p_0}\cr
&\quad+\bigg[\intop_\Omega\!\!\bigg(\intop_0^T\!\intop_0^T\!\sum_{|\alpha|+2a=[s]}\hskip-10pt {|D_x^\alpha\partial_t^au(x,t)-D_x^\alpha\partial_{t'}^au(x,t')|^{p_0}\over|t-t'|^{1+p_0({s\over 2}-[{s\over 2}])}}\d t\d t'\bigg)^{p/p_0}\hskip-5pt \d x\bigg]^{1/p},\cr}
$$
where $a\in\N\cup\{0\}$, $[s]$ is the integer part of $s$ and $D_x^\alpha$ denotes the partial derivative in the spatial variable $x$ corresponding to multiindex $\alpha$. For $s$ odd the last but one term in the above norm vanishes whereas for $s$ even the last two terms vanish. We also use notation $L_p(\Omega^T)=L_{p,p}(\Omega^T)$, $W_p^{s,s/2}(\Omega^T)=W_{p,p}^{s,s/2}(\Omega^T)$.
\item[3.] $B_{p,p_0}^l(\Omega)$, $l\in\R_+$, $p,p_0\in[1,\infty)$ -- the Besov space with the finite norm
$$
\|u\|_{B_{p,p_0}^l(\Omega)}=\|u\|_{L_p(\Omega)}+\bigg(\sum_{i=1}^n\intop_0^\infty {\|\Delta_i^m(h,\Omega)\partial_{x_i}^ku\|_{L_p(\Omega)}^{p_0}\over h^{1+(l-k)p_0}}\d h\bigg)^{1/p_0},
$$
where $k\in\N\cup\{0\}$, $m\in\N$, $m>l-k>0$, $\Delta_i^j(h,\Omega)u$, $j\in\N$, $h\in\R_+$ is the finite difference of the order $j$ of the function $u(x)$ with respect to $x_i$ with
$$\eqal{
&\Delta_i^1(h,\Omega)u=\Delta_i(h,\Omega)\cr
&=u(x_1,\dots,x_{i-1},x_i+h,x_{i+1},\dots,x_n)-u(x_1,\dots,x_n),\cr
&\Delta_i^j(h,\Omega)=\Delta_i(h,\Omega)\Delta_i^{j-1}(h,\Omega)u\quad {\rm and}\quad \Delta_i^j(h,\Omega)u=0\cr
&{\rm for}\quad x+jh\not\in\Omega.\cr}
$$
It was proved in \cite{G} that the norms of the Besov space $B_{p,p_0}^l(\Omega)$ are equivalent for different $m$ and $k$ satisfying the condition $m>l-k>0$.
\end{itemize}
\end{definition}

We need the following interpolation lemma.

\begin{lemma}(Anisotropic interpolation, see\cite[Ch. 4, Sect. 18]{BIN})\label{l2.12}
Let $u\in W_{p,p_0}^{s,s/2}(\Omega^T)$, $s\in\R_+$, $p,p_0\in[1,\infty]$, $\Omega\subset\R^3$. Let $\sigma\in\R_+\cup\{0\}$, and
$$
\varkappa={3\over p}+{2\over p_0}-{3\over q}-{2\over q_0}+|\alpha|+2a+\sigma<s.
$$
Then $D_x^\alpha\partial_t^au\in W_{q,q_0}^{\sigma,\sigma/2}(\Omega^T)$, $q\ge p$, $q_0\ge p_0$ and there exists $\varepsilon\in(0,1)$ such that
$$
\|D_x^\alpha\partial_t^au\|_{W_{q,q_0}^{\sigma,\sigma/2}(\Omega^T)}\le\varepsilon^{s-\varkappa} \|u\|_{W_{p,p_0}^{s,s/2}(\Omega^t)}+c\varepsilon^{-\varkappa}\|u\|_{L_{p,p_0}(\Omega^t)}.
$$
We recall from \cite{B} the trace and the inverse trace theorems for Sobolev spaces with a mixed norm.
\end{lemma}

\begin{lemma}\label{l2.13}
(traces in $W_{p,p_0}^{s,s/2}(\Omega^T)$, see \cite{B})
\begin{itemize}
\item[(i)] Let $u\in W_{p,p_0}^{s,s/2}(\Omega^t)$, $s\in\R_+$, $p,p_0\in(1,\infty)$. Then $u(x,t_0)=u(x,t)|_{t=t_0}\!$ for $t_0\in[0,T]$ belongs to $B_{p,p_0}^{s-2/p_0}(\Omega)$, and
$$
\|u(\cdot,t_0)\|_{B_{p,p_0}^{s-2/p_0}(\Omega)}\le c\|u\|_{W_{p,p_0}^{s,s/2}(\Omega^T)},
$$
where $c$ does not depend on $u$.
\item[(ii)] For given $\tilde u\in B_{p,p_0}^{s-2/p_0}(\Omega)$, $s\in\R_+$, $s>2/p_0$, $p,p_0\in(1,\infty)$, there exists a function $u\in W_{p,p_0}^{s,s/2}(\Omega^t)$ such that $u|_{t=t_0}=\tilde u$ for $t_0\in[0,T]$ and
$$
\|u\|_{W_{p,p_0}^{s,s/2}(\Omega^T)}\le c\|\tilde u\|_{B_{p,p_0}^{s-2/p_0}(\Omega)},
$$
where constant $c$ does not depend on $\tilde u$.
\end{itemize}
\end{lemma}

We need the following imbeddings between Besov spaces

\begin{lemma}(see \cite[Th. 4.6.1]{T})\label{l2.14}
|Let $\Omega\subset\R^n$ be an arbitrary domain.
\begin{itemize}
\item[(a)] Let $s\in\R_+$, $\varepsilon>0$, $p\in(1,\infty)$, and $1\le q_1\le q_2\le\infty$. Then
$$
B_{p,1}^{s+\varepsilon}(\Omega)\subset B_{p,\infty}^{s+\varepsilon}(\Omega)\subset B_{p,q_1}^s(\Omega)\subset B_{p,q_2}^s(\Omega)\subset B_{p,1}^{s-\varepsilon}(\Omega)\subset B_{p,\infty }^{s-\varepsilon}(\Omega).
$$
\item[(b)] Let $\infty>q\ge p>1$, $1\le r\le\infty$, $0\le t\le s<\infty$ and
$$
t+{n\over p}-{n\over q}\le s.
$$
Then $B_{p,r}^s(\Omega)\subset B_{q,r}^t(\Omega)$.
\end{itemize}
\end{lemma}

\begin{lemma}(see \cite[Ch. 4, Th. 18.8]{BIN})\label{l2.15}
Let $1\le\theta_1<\theta_2\le\infty$. Then
$$
\|u\|_{B_{p,\theta_2}^l(\Omega)}\le c\|u\|_{B_{p,\theta_1}^l(\Omega)},
$$
where $c$ does not depend on $u$.
\end{lemma}

\begin{lemma}(see \cite[Ch. 4, Th. 18.9]{BIN})\label{l2.16}
Let $l\in\N$ and $\Omega$ satisfy the $l$-horn condition.\\
Then the following imbeddings hold
$$\eqal{
&\|u\|_{B_{p,2}^l(\Omega)}\le c\|u\|_{W_p^l(\Omega)}\le c\|u\|_{B_{p,p}^l(\Omega)},\ \ &1\le p\le 2,\cr
&\|u\|_{B_{p,p}^l(\Omega)}\le c\|u\|_{W_p^l(\Omega)}\le c\|u\|_{B_{p,2}^l(\Omega)},\ \ &2\le p<\infty,\cr
&\|u\|_{B_{p,\infty}^l(\Omega)}\le c\|u\|_{W_p^l(\Omega)}\le c\|u\|_{B_{p,1}^l(\Omega)},\ \ &1\le p\le\infty.\cr}
$$
\end{lemma}

Consider the nonstationary Stokes system in $\Omega\subset\R^3$:
\[
\begin{split}
v_t-\nu\Delta v+\nabla p&=f,\\
\div\, v&=0
\end{split}
\]
with the boundary conditions \eqref{bcs} and given initial condition $v(0)$.

\begin{lemma}[see \cite{MS}]\label{l2.17}
Assume that $f\in L_{q,r}(\Omega^T)$, $v(0)\in B_{q,r}^{2-2/r}(\Omega)$, $r,q\in(1,\infty)$. Then there exists a unique solution to the above system such that $v\in W_{q,r}^{2,1}(\Omega^T)$, $\nabla p\in L_{q,r}(\Omega^T)$ with the following estimate
\begin{equation}\eqal{
&\|v\|_{W_{q,r}^{2,1}(\Omega^T)}+\|\nabla p\|_{L_{q,r}(\Omega^t)}\le c(\|f\|_{L_{q,r}(\Omega^T)}+\|v(0)\|_{B_{r,q}^{2-2/r}(\Omega)}).\cr}
\label{2.25}
\end{equation}
\end{lemma}

\subsection{Proof of \eqref{7.1}}

We show \eqref{7.1} in the following series of lemmas.

\begin{lemma}\label{l7.1}
Suppose that $X(t)<\infty$, i.e. that 
\begin{equation}
\|\Phi\|_{V(\Omega^t)}+\|\Gamma\|_{V(\Omega^t)}\le\phi_1,
\label{7.2}
\end{equation}
where $\phi_1$ depends only on the initial data and the forcing. Assume that 
$$
f\in W_2^{2,1}(\Omega^t),\quad v(0)\in W_2^3(\Omega).
$$
Then
\begin{equation}
\|v\|_{W_2^{4,2}(\Omega^t)}+\|\nabla p\|_{W_2^{2,1}(\Omega^t)}\le\phi(\phi_1,\|f\|_{W_2^{2,1}(\Omega^t)},\|v(0)\|_{H^3(\Omega)}).
\label{7.3}
\end{equation}
\end{lemma}

\begin{proof}
From (\ref{7.2}) we have
\begin{equation}
\|\Gamma\|_{V(\Omega^t)}\le\phi_1.
\label{7.4}
\end{equation}
Lemma \ref{l2.7} and Section \ref{s3} imply
\begin{equation}
\|\psi_1\|_{2,\infty,\Omega^t}\le c\phi_1,
\label{7.5}
\end{equation}
\begin{equation}
\|\psi_1\|_{3,2,\Omega^t}\le c\phi_1.
\label{7.6}
\end{equation}
From (\ref{1.22}) the following relations hold
\begin{equation}
v_r=-r\psi_{1,z},\quad v_z=2\psi_1+r\psi_{1,r}.
\label{7.7}
\end{equation}
Hence (\ref{7.5}) and $R$ finite imply
\begin{equation}
\|v_r\|_{1,2,\infty,\Omega^t}+\|v_z\|_{1,2,\infty,\Omega^t}\le c\phi_1.
\label{7.8}
\end{equation}
To increase regularity of $v$ we consider the Stokes problem
\begin{equation}\eqal{
&v_{,t}-\nu\Delta v+\nabla p=-v'\cdot\nabla v+f\quad &{\rm in}\ \ \Omega^T,\cr
&\divv v=0\quad &{\rm in}\ \ \Omega^T,\cr
&v\cdot\bar n=0,\ \ v_{z,r}=0,\ \ v_\varphi=0\quad &{\rm on}\ \ S_1^T,\cr
&v\cdot\bar n=0,\ \ v_{r,z}=0,\ \ v_{\varphi,z}=0\quad &{\rm on}\ \ S_2^T,\cr
&v|_{t=0}=v(0)\quad &{\rm in}\ \ \Omega,\cr}
\label{7.9}
\end{equation}
where $v'=v_r\bar e_r+v_z\bar e_z$.

From (\ref{7.8}) we have
\begin{equation}
|v'|_{6,\infty,\Omega^t}\le c\phi_1.
\label{7.10}
\end{equation}
For solutions to (\ref{7.9}) the following energy estimate holds
\begin{equation}
\|v\|_{V(\Omega^t)}\le c(|f|_{2,\Omega^t}+|v(0)|_{2,\Omega})\equiv d_1.
\label{7.11}
\end{equation}
Hence, (\ref{7.11}) yields
\begin{equation}
|\nabla v|_{2,\Omega^t}\le d_1.
\label{7.12}
\end{equation}
Estimates (\ref{7.10}) and (\ref{7.12}) imply
\begin{equation}
|v'\cdot\nabla v|_{3/2,2,\Omega^t}\le c\phi_1d_1.
\label{7.13}
\end{equation}
Applying \cite{MS} to (\ref{7.9}) yields
\begin{equation}\eqal{
&\|v\|_{W_{{3\over 2},2}^{2,1}(\Omega^t)}+|\nabla p|_{3/2,2,\Omega^t}\cr
&\le c(|f|_{3/2,2,\Omega^t}+\|v(0)\|_{B_{3/2,2}^1(\Omega)}+\phi_1d_1)\equiv d_2.\cr}
\label{7.14}
\end{equation}
In view of the imbedding (see \cite[Ch. 3, Sect. 10]{BIN})
\begin{equation}
|\nabla v|_{5/2,\Omega^t}\le c\|v\|_{W_{3/2,2}^{2,1}(\Omega^t)}
\label{7.15}
\end{equation}
and (\ref{7.10}) we derive that
\begin{equation}|v'\cdot\nabla v|_{{30\over 17},{5\over 2},\Omega^t}\le c\phi_1d_2.
\label{7.16}
\end{equation}
Then applying again \cite{MS} to problem (\ref{7.9}) yields
\begin{equation}\eqal{
&\|v\|_{W_{{30\over 17},{5\over 2}}^{2,1}(\Omega^t)}+|\nabla p|_{{30\over 17},{5\over 2},\Omega^t}\cr
&\le c(|f|_{{30\over 17},{5\over 2},\Omega^t}+\|v(0)\|_{B_{{30\over 17},{5\over 2}}^{2-4/5}(\Omega)}+\phi_1d_2)\equiv d_3.\cr}
\label{7.17}
\end{equation}
In view of the imbedding (see \cite[Ch. 3, Sect. 10]{BIN})
\begin{equation}
|\nabla v|_{{10\over 3},\Omega^t}\le c\|v\|_{W_{{30\over 17},{5\over 2}}^{2,1}(\Omega^t)}
\label{7.18}
\end{equation}
and (\ref{7.10}) we have
\begin{equation}
|v'\cdot\nabla v|_{{15\over 7},{10\over 3},\Omega^t}\le c\phi_1d_3.
\label{7.19}
\end{equation}
Applying \cite{MS} to (\ref{7.9}) implies
\begin{equation}\eqal{
&\|v\|_{W_{{15\over 7},{10\over 3}}^{2,1}(\Omega^t)}+|\nabla p|_{{15\over 7},{10\over 3},\Omega^t}\cr
&\le c(|f|_{{15\over 7},{10\over 3},\Omega^t}+\|v(0)\|_{B_{{15\over 7},{10\over 3}}^{2-6/10}(\Omega)}+\phi_1d_3)=d_4.\cr}
\label{7.20}
\end{equation}
Lemma \ref{l2.13} yields
\begin{equation}
\|v\|_{L_\infty(0,t;B_{{15\over 7},{10\over 3}}^{2-6/10}(\Omega))}\le c\|v\|_{W_{{15\over 7},{10\over 3}}^{2,1}(\Omega^t)}.
\label{7.21}
\end{equation}
Theorem 18.10 from \cite{BIN} gives
\begin{equation}
|v(t)|_{q,\Omega}\le c\|v\|_{B_{{15\over 7},{10\over 3}}^{7/5}(\Omega)}.
\label{7.22}
\end{equation}
The estimate holds for any finite $q$ because it satisfies the relation $7/5\ge 7/5-3/q$.

Next, we use the imbedding (see \cite[Ch. 3, Sect. 10]{BIN})
\begin{equation}
|\nabla v|_{5,\Omega^t}\le c\|v\|_{W_{{15\over 7},{10\over 3}}^{2,1}(\Omega^t)}.
\label{7.23}
\end{equation}
From (\ref{7.22}) and (\ref{7.23}) we have
\begin{equation}
|v\cdot\nabla v|_{5',\Omega^t}\le cd_4^2,
\label{7.24}
\end{equation}
where $5'<5$ but it is arbitrary close to 5.

In view of (\ref{7.24}) and \cite{MS} we have
\begin{equation}\eqal{
&\|v\|_{W_{5'}^{2,1}(\Omega^t)}+|\nabla p|_{5',\Omega^t}\le c(|f|_{5',\Omega^t}+\|v(0)\|_{W_{5'}^{2-2/5'}(\Omega)}+d_4^2)\equiv d_5.\cr}
\label{7.25}
\end{equation}
From (\ref{7.25}) it follows that $v\in L_\infty(\Omega^t)$ and $\nabla v\in L_q(\Omega^t)$ for any finite $q$.

Then
$$
|\nabla(v'\cdot\nabla v)|_{5',\Omega^t}\le cd_5^2
$$
and
$$
|\partial_t^{1/2}(v'\cdot\nabla v)|_{5',\Omega^t}\le cd_5^2,
$$
where $\partial_t^{1/2}$ denotes the fractional partial derivative in time.

Then \cite{MS} implies
\begin{equation}\eqal{
&\|v\|_{W_{10/3}^{3,3/2}(\Omega^t)}+\|\nabla p\|_{W_{10/3}^{1,1/2}(\Omega^t)}\cr
&\le c(\|f\|_{W_{10/3}^{1,1/2}(\Omega^t)}+\|v(0)\|_{W_{10/3}^{3-2/5'}(\Omega)}+d_5^2)\equiv d_6.\cr}
\label{7.26}
\end{equation}
Continuing the considerations yields
\begin{equation}
\|v\|_{W_2^{4,2}(\Omega^t)}+\|\nabla p\|_{W_2^{2,1}(\Omega^t)}\le c(\|f\|_{W_2^{2,1}(\Omega^t)}+ \|v(0)\|_{W_2^3(\Omega)}+d_6^2).
\label{7.27}
\end{equation}
This implies (\ref{7.3}) and ends the proof.
\end{proof}

\section*{Conflict of interest statement}
The authors report there are no competing interests to declare. 

\section*{Data availability statement}
The authors report that there is no data associated with this work.

\bibliographystyle{amsplain}

\end{document}